\documentclass[10pt,a4paper,reqno]{article}
\usepackage{graphicx} %,bbold,bbm,mathbbol,
\usepackage[english]{babel} %for biblatex
\usepackage{csquotes} %for biblatex
\usepackage{filecontents} %for biblatex
\usepackage[dvipsnames]{xcolor}
\usepackage{color}
\usepackage[colorinlistoftodos]{todonotes}
\usepackage{tensor}
\usepackage{thmtools}
\usepackage{microtype}

\usepackage{imakeidx}
\makeindex %for index in the document

\usepackage[
backend=biber,
hyperref=true,
backref=true,
isbn=false,
doi=true,
url=false,
natbib=true,
eprint=true,
useprefix=true,
maxcitenames=99,
maxbibnames=99,  
maxalphanames=99, 
minalphanames=99,
safeinputenc,
style=alphabetic,
citestyle=alphabetic,
block=space,
datamodel=ext-eprint,
]{biblatex}
 \usepackage[
 hypertexnames = false,
 colorlinks    = true,
 citecolor     = teal,
 linkcolor     = blue,
 urlcolor      = blue,
 linktocpage = true,
 hyperfootnotes=true,
 breaklinks
 ]{hyperref}

\renewbibmacro{in:}{} %%% for removing In : in biblatex %%%%%

\DeclareSourcemap{
	\maps[datatype=bibtex]{
		\map{
			\step[fieldsource=pmid, fieldtarget=pubmed]
		}
	}
}

\makeatletter
\DeclareFieldFormat{arxiv}{%
	arXiv\addcolon\space
	\ifhyperref
	{\href{http://arxiv.org/\abx@arxivpath/#1}{%
			\nolinkurl{#1}%
			\iffieldundef{arxivclass}
			{}
			{\addspace\texttt{\mkbibbrackets{\thefield{arxivclass}}}}}}
	{\nolinkurl{#1}
		\iffieldundef{arxivclass}
		{}
		{\addspace\texttt{\mkbibbrackets{\thefield{arxivclass}}}}}}
\makeatother
\DeclareFieldFormat{pmcid}{%
	PMCID\addcolon\space
	\ifhyperref
	{\href{http://www.ncbi.nlm.nih.gov/pmc/articles/#1}{\nolinkurl{#1}}}
	{\nolinkurl{#1}}}
\DeclareFieldFormat{mrnumber}{%
	MR\addcolon\space
	\ifhyperref
	{\href{http://www.ams.org/mathscinet-getitem?mr=MR#1}{\nolinkurl{#1}}}
	{\nolinkurl{#1}}}
\DeclareFieldFormat{zbl}{%
	Zbl\addcolon\space
	\ifhyperref
	{\href{http://zbmath.org/?q=an:#1}{\nolinkurl{#1}}}
	{\nolinkurl{#1}}}
\DeclareFieldAlias{jstor}{eprint:jstor}
\DeclareFieldAlias{hdl}{eprint:hdl}
\DeclareFieldAlias{pubmed}{eprint:pubmed}
\DeclareFieldAlias{googlebooks}{eprint:googlebooks}

\renewbibmacro*{eprint}{%
	\printfield{arxiv}%
	\newunit\newblock
	\printfield{jstor}%
	\newunit\newblock
	\printfield{mrnumber}%
	\newunit\newblock
	\printfield{zbl}%
	\newunit\newblock
	\printfield{hdl}%
	\newunit\newblock
	\printfield{pubmed}%
	\newunit\newblock
	\printfield{pmcid}%
	\newunit\newblock
	\printfield{googlebooks}%
	\newunit\newblock
	\iffieldundef{eprinttype}
	{\printfield{eprint}}
	{\printfield[eprint:\strfield{eprinttype}]{eprint}}}

\makeatletter
  \def\blfootnote{\xdef\@thefnmark{}\@footnotetext}
  \makeatother

\usepackage[margin=1.7cm,includehead]{geometry}
\usepackage{amsmath,stmaryrd, upgreek}
\usepackage{amsthm,amssymb}
\usepackage{latexsym}
\usepackage{amscd}
\usepackage{mathrsfs}
\usepackage{url}
\usepackage{mathtools}
\usepackage{doi}

\usepackage[T1]{fontenc}
\usepackage{tikz-cd}
\usepackage{titlesec}

\usepackage[titletoc,toc,title]{appendix}

\makeatletter %gets rid of Contents in toc
\renewcommand\tableofcontents{%
	\@starttoc{toc}%
	
}
\makeatother

\usepackage{enumerate}
\usepackage{slashed}
\graphicspath{}
\usepackage{fancyhdr} %gives heading and author name on alternate pages

\usepackage{changepage}

\newcounter{noteCounter}
\usepackage[nameinlink]{cleveref} %for using autoref, must be written after hyperref and amspackages

\newcommand\shorttitle{$dd^{\Phi}$-lemma and cohomologies for $\S7$-manifolds} %title which appear on alternate pages
\newcommand\authors{S. Dwivedi and R. Singhal} %author name appears on alternate pages

\newcounter{commentCounter}
\titleformat{\section}    
{\normalfont\large\bfseries\center}{\thesection.}{1em}{}
\makeatletter
\newcommand*{\rom}[1]{\expandafter\@slowromancap\romannumeral #1@}
\makeatother
\newtheorem{theorem}{Theorem}[section]
\newtheorem{corollary}[theorem]{Corollary}
\newtheorem{lemma}[theorem]{Lemma}
\newtheorem{proposition}[theorem]{Proposition}

\newtheorem{question}[theorem]{Question}
\newtheorem{claim}[theorem]{Claim}

\theoremstyle{definition}
\newtheorem{definition}[theorem]{Definition}
\newtheorem{remark}[theorem]{Remark}

\newtheorem*{ack}{Acknowledgments}
\numberwithin{equation}{section}
\def\bR{\mathbb R}

\DeclareMathOperator{\dd}{d}
\newcommand{\Spin}{\mathrm{Spin}}

\newcommand{\tr}{\mathrm{tr}}

\newcommand{\R}{\mathcal{R}}
\newcommand{\sff}{\mathfrak{spin}(7)}
\newcommand{\vol}{\mathrm{vol}}

\newcommand{\ddtwo}[2]{\mathsf d^{#1}_{#2}}          % component V_{#1} -> V_{#2}
\newcommand{\ddtwos}[2]{(\mathsf d^{#1}_{#2})^{\ast}} % its adjoint

\def\pt{\partial}
\def\del{\nabla}
\def\G2{\mathrm{G}_2}
\def\g2{\varphi}
\def\S7{\mathrm{Spin}(7)}
\def\s7{\Phi}

\def\ddphi{dd^{\Phi}}
\def\dphi{d^{\Phi}}

\def\cA{\mathcal{A}}

\def\cC{\mathcal{C}}

\def\cH{\mathcal{H}}

\def\cK{\mathcal{K}}
\def\cL{\mathcal{L}}
\def\cM{\mathcal{M}}

\def\Spin7{\mathrm{Spin(7)}}

\def\SO{\mathrm{SO}}
\def\GL{\mathrm{GL}}

\def\U{\mathrm{U}}

\def\dots7{\Dot{\Phi}}

\DeclareMathOperator{\Ima}{Im}

\newcommand\xqed[1]{%
	\leavevmode\unskip\penalty9999 \hbox{}\nobreak\hfill
	\quad\hbox{#1}}
\newcommand\demo{\xqed{$\blacktriangle$}}

\usepackage{cancel}
\usepackage[all]{xy}
\usepackage{multicol}

\usepackage{charter}

\usepackage[most]{tcolorbox}
\usepackage{tikz}
\usepackage[T1]{fontenc}            % Standard package for selecting font encodings

\newtcolorbox{mydefinition}{colback=blue!5!white,colframe=blue!75!black}

\newtcolorbox{mytheorem}{colback=green!5!white,colframe=green!75!black}

\usepackage{pgfplots}
\pgfplotsset{compat=1.11}

\usepackage{pdfpages}

\begin{document}

\title{A $dd^\Phi$-Lemma and Bott–Chern-type Cohomology for Spin(7)-Manifolds}    
\author{Shubham Dwivedi and Ragini Singhal}	

\date{}

\maketitle

\begin{abstract}
\noindent
We study the properties of the $\ddphi$-operator on $8$-dimensional $\S7$-manifolds with torsion-free $\S7$-structures $\Phi$. These operators were first introduced by Harvey and Lawson in \cite{HL-intropotential}. We prove a Hodge decomposition theorem for the $dd^\Phi$-operator and obtain an analogue of the $\pt \bar{\pt}$-lemma in Kähler geometry. Using this, we define Bott--Chern-type cohomologies for $\S7$-manifolds. We relate the Bott-Chern-type cohomology spaces to the moduli space of torsion-free Spin(7)-structures and calibrated geometry of Spin(7)-manifolds. These relations naturally give rise to the notion of Cayley-positive cones. In the course of proving the results, we state and prove various identities for the exterior derivative and its decompositions into irreducible $\S7$-representations as well as identities for second order derivatives and Laplacians. The identities we prove are for any Spin(7)-structures and the specialized torsion-free ones are Spin(7)-analogoues of Kähler identities and Bryant--Harvey's identities in the $\G2$-case \cite{bryant-someremarks} and are results of independent interest.

\end{abstract}

 \begin{adjustwidth}{0.95cm}{0.95cm}
    \tableofcontents
 \end{adjustwidth}
 %\listoftodos
% %\newpage

%\let\thefootnote\relax\footnotetext{\emph{MSC (2020): 53E99, 53C29, 53C21, 53C15.}}

\section{Introduction}\label{sec:intro}

The main goal of this paper is to study properties of a second order differential operator introduced by Harvey--Lawson in \cite{HL-intropotential} for $8$-dimensional manifolds with a torsion-free Spin(7)-structure. A major motivation for this paper is the recent work of Pacini--Raffero \cite{pacini-raferro-pluripotential} where they prove results similar to those in this paper, for Calabi--Yau $6$-manifolds and $7$-dimensional manifolds with torsion-free $\G2$-structures.

\medskip

Let $M^n$ be a smooth manifold with a $G$-structure where $G\subset \GL(n, \bR)$. The action of $G$ induces a decomposition of all tensor bundles of $M$ which is governed by the representation theory of the group $G$. The space $\Omega^k(M),\ 0\leq k\leq n$ also decomposes further and as a result, the exterior derivative operator $d$ also decomposes. For example, on a complex manifold, the exterior derivative $d$ decomposes into operators $\partial$, $\bar\partial$. If the manifold has a metric which is compatible with the underlying complex structure then Hermitian Hodge theory provides Hodge decompositions for these operators and the corresponding second order differential operators. These results give rise to cohomology spaces and are one way to express them are in terms of the Dolbeault cohomology spaces $H^{p,q}(M)= \frac{\ker(\bar\partial)}{\Ima(\bar\partial)}$.

\medskip

When $G=\U(n)$ and the manifold has a K\"ahler structure, the splitting of differential forms and the decomposition of $d$ are fairly well-studied. For torsion-free $\G2$-manifolds with $G$ being the exceptional Lie group $\G2$, the decomposition of differential forms is well known and the decomposition of the exterior derivative $d$ was studied in detail by Bryant--Harvey \cite[Table 1]{bryant-someremarks}. The decompositions of the exterior derivative into irreducible components give rise to various second-order identities and relations between the Laplacian operator. This was also done by Bryant--Harvey \cite[Table 2 and 3]{bryant-someremarks}. The decompositions of the exterior derivative and the second-order operators for $\G2$-manifolds were also proved in \cite{chan-karigiannis-tsang}. 

\medskip

Since the decompositions of the exterior derivative for $\G2$-manifolds have been extremely useful in various problems, the first goal of the present paper is to study the same for $8$-dimensional manifolds with structure group Spin(7). Recall that a $8$-dimensional spin manifold $M^8$ admits a Spin(7)-structure if the structure group of its frame bundle $Fr(M)$ reduces from the Lie group $\GL(8, \bR)$ to the Lie group Spin(7) which is the double cover of $\SO(7, \bR)$. The differential geometric information of such a structure is encoded in a differential $4$-form $\Phi\in \Omega^4(M)$ which is \emph{admissible}. See \Cref{sec:prelims} for more details. 

\medskip

We denote by $(M^8, \Phi)$ a manifold which admits a Spin(7)-structure. The space of differential forms decomposes further as per irreducible Spin(7)-representations. Following Bryant \cite{bryant-someremarks}, we explicitly compute the decomposition of the exterior derivative into various irreducible components. We do this for \emph{any} Spin(7)-structure (not necessarily torsion-free) and the results are compiled in \Cref{table:dwithtorsion}. Since we will be only needing the decompositions for torsion-free Spin(7)-structures, that is, Spin(7)-structures with $d\Phi=0$, we also state the corresponding decomposition in \Cref{table:dtorsionfree}. Understanding these decompositions is a vital step towards our second main goal, which is to study the properties of Harvey--Lawson's $\ddphi$-operator. Given a manifold $M^8$ with a Spin(7)-structure $\Phi$, the $\ddphi$-operator is a second order operator and is defined as
\begin{align*}
\ddphi: \Omega^0(M)\rightarrow \Omega^4(M),\ \ddphi(f) = d(df^{\sharp}\lrcorner \Phi).
\end{align*}

\medskip

We obtain a Hodge type decomposition theorem for compact Spin(7)-manifolds $(M^8, \Phi)$ with $d\Phi=0$ using the operator $dd^\Phi$. More precisely, we prove the following theorem.

\medskip

\noindent
\textbf{\Cref{thm:Spin7Hodge}} Let $(M^8, \Phi)$ be a compact manifold with a torsion-free Spin(7)-structure $\Phi$. Then there exists an $L^2$-orthogonal decomposition of the space $K$ in \cref{eq:cohospace} as
\begin{align}\label{eq:Spin7Hodgeintro}
K=\Ima(\ddphi) \oplus \cH^4_1\oplus \cH^4_{35},
\end{align}
where $\cH^4_1$ and $\cH^4_{35}$ denote the space of harmonic forms in the spaces $\Omega^4_1$ and $\Omega^4_{35}$ respectively.

\medskip

These results are analogous to the corresponding decomposition theorem in the K\"ahler case and for $\G2$-manifolds. The latter were proved in a recent work of Pacini--Raffero \cite[Thm. 5.12]{pacini-raferro-pluripotential}. As mentioned above, a key step consists in comparing the $dd^\Phi$ operator with the corresponding decomposition of $d$ using \Cref{table:dtorsionfree}, in order to detect the correct subspace of $\Omega^4(M)$ within which to set up the Hodge decomposition.

Using \Cref{thm:Spin7Hodge}, we obtain natural analogue of the $\partial\bar\partial$-lemma in complex geometry for manifolds with torsion-free Spin(7)-structures. 

\medskip

\noindent
\textbf{\Cref{lem:Spin7delbarlemma}} Let $(M^8, \Phi)$ be a compact Spin(7)-manifold and hence $d\Phi=0$, and let $\gamma \in \Omega^4_1\oplus \Omega^4_{35}$. If $\gamma$ is globally $d$-exact, then $\gamma$ is globally $\ddphi$-exact, that is, $\gamma=\ddphi f$ for some $f\in C^{\infty}(M)$.

\medskip

In fact, the statement of the $\ddphi$-lemma in Spin(7)-geometry is precisely like the $\pt\bar{\pt}$-lemma in K\"ahler geometry. We use the Spin(7)-$\ddphi$-lemma to study cohomology of Spin(7)-manifolds and these are used to define analogues of the Bott--Chern cohomology in \Cref{def:bot-chernspin7}. More precisely, for a compact Spin(7)-manifold $(M^8, \Phi)$, we define the \textbf{Bott--Chern-type $\ddphi$-cohomology space} as
\begin{align}\label{eq:bot-cherncohintro}
H^{\Phi}(M) = \cfrac{\ker \left(\left.d\right|_{\Omega^4_{1\oplus 35}(M)}\right)}{\Ima(\ddphi)}.
\end{align}
We show in \Cref{lem:Spin7coh} that these $\ddphi$-Bott--Chern cohomology spaces are isomorphic to certain spaces of harmonic forms defined by the action of the Lie group Spin(7). 

\medskip

The paper is organized as follows. We discuss preliminaries on Spin(7)-geometry in \Cref{sec:prelims}. Harvey--Lawson's $\ddphi$-operator is described in \Cref{subsec:harveylawsonop}. Our first main result is described in \Cref{lem:ddecomp}, \Cref{table:dwithtorsion} which states the decomposition of the exterior derivative $d$ on various subspaces for arbitrary Spin(7)-structures. For the sake of completeness and ease of use throughout the paper, we re-state the table for torsion-free Spin(7)-structures in \Cref{table:dtorsionfree}. These are analogous to the torsion-free $\G2$-structures case in \cite[Table 1]{bryant-someremarks} and \cite[Figure 1]{chan-karigiannis-tsang}. This is followed by a table which contains new identities and relations between second order operators and the Laplacian on Spin(7)-manifolds. These are stated in \Cref{table:d2} and \Cref{table:d3}, respectively. We proceed to prove our other main results on $\ddphi$-lemma and the $\ddphi$-Bott--Chern cohomology in \Cref{sec:cohomology}. The $\ddphi$-lemma and the resulting $\ddphi$-cohomology naturally give rise to the notion of \emph{positive cones of a Spin(7)-structure}. We make this notion precise and explain how this is related to calibrated geometry of Spin(7)-manifolds. We discuss the notion of \textbf{Cayley-positive} $4$-forms in \Cref{def:cayleypositive} which allows us to define \textbf{Cayley-positive cone} of a torsion-free Spin(7)-structure $\Phi$ in \Cref{def:spin7cone}. We show that the space in \Cref{def:spin7cone} is indeed a pointed convex cone in \Cref{prop:conesprop} and use it to give a notion of Spin(7)-potentials in \Cref{lem:spin7potential}. These notions were first studied by Harvey--Lawson \cite{HL-intropotential, HL-positive}.

Another paper which studies cohomology of Spin(7)-manifolds is \cite{KLS} by Kawai--Lê--Schwachhöfer which uses the Frölicher-Nijenhuis bracket to define and study cohomology of Spin(7)-manifolds. Although in the torsion-free case, the $\ddphi$-operator and the Frölicher-Nijenhuis operator in \cite{KLS} agree on functions, there is no natural identification between these cohomologies.

\ack{The first author would like to thank the organizers of the AIM workshop "New directions in $\G2$-geometry", Jason Lotay and Tommaso Pacini, for the invitation to participate where the idea of this project was first formed. He thanks the participants for many useful discussions and the staffs of the institute for providing excellent working conditions. Both the authors would like to thank Tommaso Pacini for sharing his ideas about pluripotential theory on special holonomy manifolds and its applications. The first author acknowledges support by the Deutsche Forschungsgemeinschaft (DFG, German Research Foundation) under Germany’s Excellence Strategy – EXC 2121 "Quantum Universe" – 390833306. The second author is funded by the Deutsche Forschungsgemeinschaft (DFG, German Research Foundation) under Germany’s Excellence Strategy EXC 2044 – 390685587, Mathematics Münster: Dynamics–Geometry–Structure.}

\section{Preliminaries on Spin(7)-structures}\label{sec:prelims}
We discuss some preliminaries on Spin(7)-structures in this section. More details on Spin(7)-geometry can be found in \cite{joycebook} and \cite{karigiannis-spin7}. 

\medskip

Let $M^8$ be a smooth manifold. A Spin(7)-structure on $M$ is the reduction of the frame bundle of $M$ from the Lie group $\GL(8,\bR)$ to the Lie group Spin(7). Such a reduction is equivalent to the existence of a smooth $4$-form $\Phi\in \Omega^4(M)$ which is \emph{admissible}. An admissible $4$-form induces a Riemannian metric and an orientation and hence a Hodge star operator $*$.The form $\Phi$ is self-dual,
\begin{align*}
*\Phi=\Phi.
\end{align*}

\begin{definition}\label{def:tfSpin7str}
    Let $\del$ be the Levi-Civita connection of the metric $g$. The pair $(M^8, \s7)$ is a \emph{Spin(7)-manifold} if $\del \s7=0$. This is a non-linear partial differential equation for $\s7$, since $\del$ depends on $g$, which in turn depends non-linearly on $\s7$. A Spin(7)-manifold has Riemannian holonomy contained in the subgroup $\S7\subset \SO(8)$. Such a parallel Spin(7)-structure is also called \emph{torsion-free}. The condition $\del \Phi=0$ is equivalent to $d\Phi=0$. 
\end{definition}

\subsection{Decomposition of the space of forms}\label{subsec:formdecomp}

The existence of a Spin(7)-structure $\s7$ induces a decomposition of the space of differential forms on $M$  into irreducible Spin(7)-representations. We have the following orthogonal decompositions, with respect to $g$:
\begin{align*}
\Omega^2=\Omega^2_{7}\oplus \Omega^2_{21},\ \ \ \ \ \ \ \Omega^3=\Omega^3_{8}\oplus \Omega^3_{48}, \ \ \ \ \ \ \ \ \  \Omega^4=\Omega^4_{1}\oplus \Omega^4_{7}\oplus \Omega^4_{27}\oplus \Omega^4_{35},
\end{align*}
where $\Omega^k_l$ has point wise dimension $l$. Explicitly, $\Omega^2$ and $\Omega^3$ are described as follows:
\begin{align}\label{eq:Omega2decomp1}
\Omega^2_7=\{\beta \in \Omega^2 \mid *(\s7 \wedge \beta)=3\beta\}, \ \ \ \ \ \ \ \Omega^2_{21}&=\{ \beta\in \Omega^2 \mid *(\s7\wedge \beta)=-\beta\}, 
\end{align}
and
\begin{align}\label{eq:Omega3decomp1}
\Omega^3_8=\{ X\lrcorner \s7 \mid X\in \Gamma(TM)\},\ \ \ \ \ \ \ \ \ \ \ \ \Omega^3_{48}=\{ \gamma \in \Omega^3\mid \gamma \wedge \s7 =0\}. 
\end{align}
Using \cref{eq:Omega2decomp1}, we have for $\beta\in \Omega^2(M)$,
\begin{align}\label{eq:Omega2decomp2}
\beta_{ij}\in \Omega^2_7 \iff \beta^{ab}\s7_{abij}=6\beta_{ij},\ \ \ \ \text{and}\ \ \ \ 
\beta_{ij}\in \Omega^2_{21} \iff \beta^{ab}\s7_{abij}=-2\beta_{ij},
\end{align}
where we are using the underlying metric to raise the indices.

% \begin{remark}
% A convention different than ours is also prevalent in the literature. In this convention, $\Omega^2_7$ and $\Omega^2_{21}$ are the $+3$ and $-1$ eigenspaces of the map $\beta\mapsto *(\Phi\wedge \beta),$ respectively. As a result, the constants on the right hand sides of \cref{eq:27decomp2} and \cref{eq:221decomp2} are $+6$ and $-2$ respectively.  
% \end{remark}

\noindent
For $\gamma\in \Omega^3(M)$,
\begin{align}\label{eq:Omega3decomp2}
\gamma_{ijk}\in \Omega^3_8 \iff \gamma_{ijk}=X_l\s7_{ijkl}\ \textup{for\ some}\ X\in \Gamma(TM),\ \ \ \ \gamma_{ijk}\in \Omega^3_{48} \iff \gamma_{ijk}\s7_{ijkl}=0. 
\end{align}
If $\pi_7$ and $\pi_{21}$ are the projection operators on $\Omega^2$, it follows from \cref{eq:Omega2decomp2} that 
\begin{align}
\pi_7(\beta)_{ij}&=\frac 14\beta_{ij}+\frac 18\beta_{ab}\s7_{abij}, \label{eq:pi7}\\
\pi_{21}(\beta)_{ij}&=\frac 34\beta_{ij}-\frac 18\beta_{ab}\s7_{abij}. \label{eq:pi21}
\end{align}
We will be using these equations throughout the paper.
Finally, for $\beta_{ij}\in \Omega^2_{21}$,
\begin{align}
    \beta_{ab}\s7_{bpqr}&=\beta_{pi}\s7_{iqra}+\beta_{qi}\s7_{irpa}+\beta_{ri}\s7_{ipqa}, \label{eq:221prop}   
\end{align}
which can be used to show that $\Omega^2_{21}\equiv \mathfrak{spin}(7)$ is the Lie algebra of Spin(7), with the commutator of matrices
\begin{align*}
    [\alpha, \beta]_{ij}=\alpha_{il}\beta_{lj}-\alpha_{jl}\beta_{li}.
\end{align*}

\medskip

\noindent
To describe $\Omega^4$ in local orthonormal frame, we use the $\diamond$ operator (cf. \cite{dgk-isometric, dle-isometric, Spin7-flow_Dwivedi}). Given $A\in \Gamma(T^*M\otimes TM)$, define
\begin{align}
\label{eq:diadefn1}
    A\diamond \s7= \frac{1}{24}(A_{i}^p\s7_{pjkl}+A_{j}^p\s7_{ipkl}+A_{k}^p\s7_{ijpl}+A_{l}^p\s7_{ijkp})e^i\wedge e^j\wedge e^k\wedge e^l,    
\end{align}
and hence 
\begin{align}
\label{eq:diadefn2}
    (A\diamond \s7)_{ijkl}=  A_{i}^{\, p}\s7_{pjkl}+A_{j}^{\, p}\s7_{ipkl}+A_{k}^{\, p}\s7_{ijpl}+A_{l}^{\, p}\s7_{ijkp}. 
\end{align}
In general for any $k$-form $\omega$ the operator $\diamond \ \omega$ defines the infinitesimal $\mathrm{GL}(8)$-action on forms. So for $A\in \mathfrak{gl}(8,\bR)$ in an orthonormal frame $\{e_i,i=1,\ldots,8\}$ 
\begin{align}\label{eqn:gen_diamond}
    A \diamond \omega &= \sum_{i=1}^8 e^i\wedge (A(e_i)\lrcorner \omega).
\end{align}
Recall that $
    \Gamma(T^*M\otimes TM)=\Omega^0\oplus S_0\oplus \Omega^2$, where $S_0$ is the space of trace-free symmetric $2$-tensors and $\Omega^2$ further splits orthogonally into \cref{eq:Omega2decomp1}, so
\begin{align}
\label{eq:splitting TM* x TM}
 \Gamma(T^*M\otimes TM)=\Omega^0\oplus S_0 \oplus \Omega^2_7\oplus \Omega^2_{21}.   
\end{align}
With respect to this splitting, we can write $A=\frac 18 (\tr A)g+A_{35}+A_7+A_{21}$ where $A_{35}$ is a symmetric traceless $2$-tensor. 
The diamond contraction \cref{eq:diadefn2} defines a linear map $A\mapsto A\diamond \s7$, from $\Omega^0\oplus S_0 \oplus \Omega^2_7\oplus \Omega^2_{21}$ to $\Omega^4(M)$. The kernel of the map $A\mapsto A\diamond \s7$ is isomorphic to the subspace $\Omega^2_{21}$. The remaining three summands $\Omega^0,\ S_0$ and $\Omega^2_7$ are mapped isomorphically onto the subspaces $\Omega^4_1,\ \Omega^4_{35}$ and $\Omega^4_7$ respectively. We also note that 
\begin{align}\label{eq:sdasdforms}
\Omega^4_{1\oplus7\oplus 27}(M) = \{\gamma \in \Omega^4 \mid *\gamma = \gamma\} = \text{self-dual\ }4\text{-forms},\ \ \Omega^4_{35}(M)=\{\gamma \in \Omega^4 \mid *\gamma =- \gamma\} = \text{anti-self-dual\ }4\text{-forms}.
\end{align}

\medskip

\noindent
Before we discuss the torsion of a Spin(7)-structure, we note some contraction identities involving the $4$-form $\s7$. In local coordinates $\{x^1, \cdots, x^8\}$, the $4$-form $\s7$ is
\begin{align*}
\s7=\frac{1}{24}\s7_{ijkl}\ dx^i\wedge dx^j\wedge dx^k\wedge dx^l    
\end{align*}
where $\s7_{ijkl}$ is totally skew-symmetric. We have the following identities

\begin{align}
\s7_{ijkl}\s7_{ab}^{\, \, \, \, kl}&=6g_{ia}g_{jb}-6g_{ib}g_{ja}+4\s7_{ijab}, \label{eq:impiden2} \\
    \s7_{ijkl}\s7_{a}^{\, \, \, \, jkl}&=42g_{ia}, \label{eq:impiden3} \\
    \s7_{ijkl}\s7^{ijkl}&=336 . \label{eq:impiden4}
\end{align}
We also have contraction identities involving $\del \s7$ and $\s7$
\begin{align}
(\del_m\s7_{ijkl})\s7_{ab}^{\, \, \, kl}&=-\s7_{ij}^{\, \, \, \, kl}(\del_m\s7_{abkl})+4\del_m\s7_{ijab} \label{eq:impiden5}\\
(\del_m\s7_{ijkl})\s7_{a}^{\, \, \, jkl}&=-\s7_{i}^{\, \, \, jkl}(\del_m\s7_{ajkl}) \label{eq:impiden6} \\
(\del_m\s7_{ijkl})\s7^{ijkl}&=0. \label{eq:impiden7}
\end{align}

%%%%%%%%%%%%%%%%%%%%
%%%%%%%%%%%%%%%%%%%
\medspace

The \emph{intrinsic torsions} of a Spin(7)-structure are differential forms which lie $\Omega^5_8\oplus \Omega^5_{48}$,
\begin{align}\label{eq:intrinsictorsions}
d\Phi = T_8 \wedge \Phi + * T_{48}.
\end{align}
For torsion free Spin(7)-manifold, $d\Phi=0$ is closed (and thus, co-closed) which implies 
\begin{align*}
T_8=0, \qquad T_{48}=0.    
\end{align*}
The \emph{torsion} of a  Spin(7)-structure $\Phi$ is a $3$-form $T$ which is given by
\begin{align}\label{eq:torsion}
T = T_8\lrcorner\s7 + T_{48}.
\end{align}

\medskip

Let $W\in \Gamma(TM)$. We have the following expression for the Lie derivative of $\Phi$ in the direction of $W$ (cf. \cite[eq. (2.39)]{Spin7-flow_Dwivedi}),
\begin{align}
\cL_W\Phi&=\left(\left(\frac 12 \cL_Wg + W\lrcorner T +(\del W)_7 \right)  \diamond \s7 \right), \label{eq:liePhiexp} 
\end{align}
where $(\del W)_7$ means the $\pi_7$-component of the $2$-tensor $\del W$.

% \medspace

\subsection{Harvey--Lawson's Hessian operator}\label{subsec:harveylawsonop}

Let $(M^8, \Phi)$ be a manifold with a Spin(7)-strcuture $\Phi$. Following \cite{HL-intropotential}, the $\dphi$-operator and the $\ddphi$-operators are defined as 

\begin{align}\label{eq:dphiop}
\dphi: \Omega^0(M) \rightarrow \Omega^3(M),\ \dphi(f) = df\lrcorner \Phi,	
\end{align}
and 
\begin{align}\label{eq:ddphiop}
\ddphi: \Omega^0(M)\rightarrow \Omega^4(M),\ \ddphi(f) = d(df\lrcorner \Phi).
\end{align}
The operator $\dphi$ was also studied in \cite{Verbitsky}. Clearly, $\dphi(f)\in \Omega^3_8(M)$.
Let us  try to understand the operator using the $\diamond$-operator. Let $Y\in \Gamma(TM)$ and let $A=\del Y$ in \cref{eq:diadefn2}, then
 \begin{align*}
 (\del Y \diamond \Phi)_{ijkl}= (\del_iY^p)\Phi_{pjkl}+(\del_jY^p)\Phi_{ipkl}+(\del_kY^p)\Phi_{ijpl}+ (\del_lY^p)\Phi_{ijkp}.
 \end{align*}
Recall that 
\begin{align*}
(\cL_Y \Phi)_{ijkl}= Y^p\del_p\Phi_{ijk}+(\del_iY^p)\Phi_{pjkl}+(\del_jY^p)\s7_{ipkl}+(\del_kY^p)\s7_{ijpl}+(\del_lY^p)\s7_{ijkp},
\end{align*}
and hence
\begin{align*}
(\del Y \diamond \Phi)_{ijk} = -Y^p\del_p\Phi_{ijkl}+(\cL_Y\Phi)_{ijkl} = d(Y\lrcorner \Phi)_{ijkl}+ (Y\lrcorner d\Phi)_{ijkl} - (\del_Y \Phi)_{ijkl}.
\end{align*}
We can also use the fact that 
\begin{align*}
\cL_Y\Phi = \left(\frac 12 \cL_Y g  + Y\lrcorner T + (\del Y)_7\right)\diamond \Phi,
\end{align*}
implying that
\begin{align*}
(\del Y \diamond \Phi)_{ijkl} = -Y^p\del_p\Phi_{ijkl} + \left(\left(\frac 12 \cL_Y g  + Y\lrcorner T + (\del Y)_7\right)\diamond \Phi\right)_{ijkl},
\end{align*}

\medskip

We have two cases:

\medskip

\noindent
{\bf{1.}} Suppose $\Phi$ is a calibration, i.e., $d\Phi=0$ in which case the Spin(7)-structure is torsion-free, and $Y$ is a vector field then we have
\begin{align}
(\del Y\diamond \Phi) = d(Y\lrcorner \g2).
\end{align}
In the case when $Y=df$, this becomes
\begin{align}
(\del \del f\diamond \Phi) = d(df\lrcorner \Phi) = \ddphi f,
\end{align}
is precisely the Harvey--Lawson $\ddphi$-operator in \cref{eq:ddphiop}.

\medspace

\noindent
{\bf{2.}} When $\Phi$ is arbitrary then 
\begin{align}
(\del Y\diamond \g2) = d(Y\lrcorner \Phi)+Y\lrcorner d\Phi - \del_{Y}\Phi,
\end{align}
which in the case of $Y=\del f$ becomes
\begin{align}
(\del \del f \diamond \Phi)= d(df\lrcorner \Phi)+ df\lrcorner d\Phi - \del_{\del f}\Phi,
\end{align}
thus defining the Harvey--Lawson $\Phi$-Hessian operator as

\begin{align}\label{eq:hlhessiangeneral}
\cH^{\Phi}(f):\Omega^0(M) \rightarrow \Omega^4(M), \qquad \cH^{\Phi}(f) =  d(df\lrcorner \Phi)+ df\lrcorner d\Phi - \del_{\del f}\Phi.
\end{align}

These are the same relations as in \cite[\textsection 0]{HL-intropotential}, we just write them using the $\diamond$-operator.

\section{Decomposition of the exterior derivative}\label{sec:decompositionofd}

\subsection{First-order identities}

In this section, we use the decomposition of forms on $(M^8, \Phi)$ discussed in \Cref{subsec:formdecomp} to state and prove the decompositions of the extrior derivative $d$ into various components. We use the following notation. Denote by
\begin{align}
\dd^p_q:\Omega^k_p\hookrightarrow \Omega^k \xrightarrow{d} \Omega^{k+1}_q,
\end{align}
whenever this operation makes sense. The only term in the \Cref{table:dwithtorsion} which doesn't follow this convention is the $\dd^{48}_{21}$ term in the 12th row. That term is meant to denote the $\Omega^2_{21}$-component of the $2$-form $d^*\gamma_{48}, \gamma_{48}\in \Omega^3_{48}(M)$. The next theorem computes $\dd^{p}_q$ for all possible $p$ and $q$ in the Spin(7)-case for an arbitrary Spin(7)-structure. 

\begin{theorem}\label{lem:ddecomp}
    Let $M^8$ be a smooth manifold with a Spin(7)-structure $\Phi$. Let $d\Phi=T_8\wedge\Phi+*T_{48}$ be the intrinsic torsion forms. Then
for all $p, q \in \{1, 7, 8, 21, 27,35,48\}$ there exist a first-order differential operator $\dd^p_q \colon \Omega^k_p \to \Omega^{k+1}_q$ so that
the following exterior derivative formulae hold for all $f \in \Omega^0_1, \alpha \in \Omega^1_8, \beta=\beta_7+\beta_{21}\in \Omega^2_7\oplus\Omega^2_{21},\eta=\eta_{27}+\eta_{35}\in \Omega^4_{27}\oplus\Omega^4_{35},$ and $\gamma_{48} \in \Omega^3_{48}$:
\end{theorem}

\begin{table}[htp!]
\resizebox{\textwidth}{!}{%
  \renewcommand{\arraystretch}{1.4}%
  \setlength{\arraycolsep}{3pt}%
 $ \begin{array}{l c c c c c c c c}
  \hline
  d f                  & = &                                        &                             & \dd_8^{\,1} f                  &                            &                       &                       &           \\
  \hline 
  d(f\Phi)             & = &                                        &                             & (\dd_8^{\,1} f+f T_8)           &                            &                       &                       & f*T_{48}  \\
  \hline
  d\alpha              & = &                                        & \dd_7^{\,8}\alpha           &                                & \dd_{21}^{\,8}\alpha       &                       &                       &           \\ \hline
  d*(\alpha\wedge\Phi) & = & -\frac12(\dd_1^{\,8}\alpha+g(\alpha,T_8)) & -\tfrac12 \dd_7^{\,8}\alpha &                        &                            &             *(\alpha \wedge T_{48})_{27}         & \dd^8_{35}\alpha        &           \\ \hline
  d(\alpha\wedge\Phi)  & = &                                        & 3*\dd_7^{\,8}\alpha - (\alpha\wedge \dd\Phi)_7 &              & -*\dd_{21}^{\,8}\alpha -(\alpha\wedge d\Phi)_{21} & & &           \\ \hline
  \dd(*\alpha)           & = & -\,\dd_1^{\,8}\alpha                   &                             &                                &                            &                       &                       &           \\ \hline
  d\beta_7             & = &                                        &                             & \tfrac17(3\dd_8^{\,7}\beta-*(\beta_7\wedge \dd\Phi)) &   &                       &                       & \dd_{48}^{\,7}\beta  \\ \hline
  d\beta_{21}          & = &                                        &                             &  -\tfrac17\,(\dd_{8}^{21}\beta+*(\beta_{21}\wedge \dd\Phi)) &                 &                       &                       & \dd_{48}^{\,21}\beta \\ \hline
  d(\beta_7\wedge\Phi)    & = &                                     &                             & -3*\dd_8^{\,7}\beta            &                            &                       &                       &           \\ \hline
  d(\beta_{21}\wedge\Phi) & = &                                     &                             & *\dd_8^{\,21}\beta            &                            &                       &                       &           \\ \hline
  d\gamma_{48}         & = &    \frac{1}{14} g(\gamma_{48},T_{48})                                    & \dd_7^{\,48}\gamma          &                                &                            & \dd_{27}^{\,48}\gamma & \dd_{35}^{\,48}\gamma  &           \\ \hline
  d*\gamma_{48}        & = &                                        & \tfrac1{12}(j_\Phi(\dd_7^{\,48}\gamma)\wedge\Phi) &                       & -(\dd_{21}^{\,48}\gamma\wedge\Phi) &             &                       &           \\ \hline
  d\eta_{27}           & = &                                        &                             &              \mathcal C(T_{48},\eta_{27})                  &                            &                       &                       & \dd_{48}^{\,27}\eta  \\ \hline
  d\eta_{35}           & = &                                        &                             & \dd_8^{\,35}\eta     &                            &                       &                       & \dd_{48}^{\,35}\eta  \\ \hline
  
  \end{array} $
  }
\caption{Exterior derivative formulae for $d\Phi=T_8\wedge\Phi+*T_{48}$}
\label{table:dwithtorsion}
\end{table}

\begin{proof}
We first mention that an empty space in any of the columns of the table indicate that the contribution in the corresponding irreducible representation from $\dd^p_q$ is zero. Let $V_k$ denote the $k$-dimensional irreducible Spin(7)-representation. Then 
$$
\Omega^1_8\cong\Omega^3_8\cong\Omega^5_8\cong\Omega^7_8\cong TM\;(V_8),\qquad
\Omega^2_7\cong\Omega^4_7\cong\Omega^6_7\;(V_7),$$$$
\Omega^2_{21}\cong\Omega^6_{21}\cong\mathfrak{spin}(7),\quad
\Omega^3_{48}\cong\Omega^5_{48}\;(V_{48}),\quad
\Omega^4_1\cong\Omega^0,\quad
\Omega^4_{27}\cong S^2_0V_7,\quad
\Omega^4_{35}\cong S^2_0TM .
$$
Let \(E_q^p\subset\Lambda^pT^*M\) be an irreducible Spin(7)-subbundle. In the torsion-free case the type projections are parallel and
\[
\sigma_\xi(d_{q'}^q)=\pi_{q'}\circ(\xi\wedge\cdot):E_q^p\longrightarrow E_{q'}^{p+1}.
\] This symbol is Spin(7)-equivariant in \((\xi,\alpha)\), hence it can
be nonzero only if \(V_{q'}\) occurs in \(V_8\otimes V_q\) and in
\(\Lambda^{p+1}V_8^*\). For a Spin(7)-structure with torsion,
differentiating the type projections produces additional zero-order
terms linear in the intrinsic torsion \(W\cong V_8\oplus V_{48}\). A
torsion correction of type \(V_{q'}\) can occur only if
\(V_{q'}\subset W\otimes V_q\). We note the following decomposition of Spin(7)-representations (see, for instance, \cite[\textsection 2]{Spin7-flow_Dwivedi}): 
\begin{align}\label{eq:repdecomp}
\begin{split}
    V_8\!\otimes\!V_1&=V_8,\quad V_8\!\otimes\!V_7=V_8\oplus V_{48},\quad V_8\!\otimes\!V_8=V_1\oplus V_7\oplus V_{21}\oplus V_{35}, \quad V_8\!\otimes\!V_{21}=V_8\oplus V_{48}\oplus V_{112},\\
    V_8\otimes V_{27}&=V_{48}\oplus V_{168}, \quad V_8 \otimes V_{35} = V_8 \oplus V_{48} \oplus V_{224},\quad
V_8\!\otimes\!V_{48}\supset V_7\oplus V_{21}\oplus V_{27}\oplus V_{35}, \\
    V_7 \otimes V_{48} &= V_8 \oplus V_{48} \oplus V_{112} \oplus V_{168},  \quad V_{21} \otimes V_{48} \supset V_8 \oplus V_{48}, \quad
V_{27} \otimes V_{48} \supset V_8 \oplus V_{48}, \\ V_{35} \otimes V_{48} &\supset V_8 \oplus V_{48}, \qquad V_{48} \otimes V_{48} \supset V_1 \oplus V_7 \oplus V_{21} \oplus V_{27} \oplus V_{35} \oplus V_{48}.
\end{split}
\end{align}

%================== f in Omega^0 ====================
When $f \in \Omega^0$
\[
d f = \dd_8^{\,1} f, \qquad
d(f\Phi) = \dd_8^{\,1} f \wedge \Phi +f d\Phi =  (\dd_8^{\,1} f+fT_8) \wedge \Phi + f*T_{48} .
\]This gives the first and the second row.

%==================== alpha in Omega^1 ====================
For $\alpha \in \Omega^1$

\[
d\alpha = (d\alpha)_7 + (d\alpha)_{21}
         = \dd_7^{\,8}\alpha + \dd_{21}^{\,8}\alpha .
\]

we can compute from the definitions of $\Lambda^2_7,\Lambda^2_{21}$ 
\[
(d\alpha)_7   = \tfrac14\big(*(d\alpha\wedge\Phi) + d\alpha\big) = \dd_7^{\,8}\alpha,
\]
\[
(d\alpha)_{21} = -\tfrac14\big(*(d\alpha\wedge\Phi) - 3 d\alpha\big) = \dd_{21}^{\,8}\alpha,
\]thus providing the entries in the third row.

Similarly
$d(\alpha\wedge\Phi)\in\Lambda^6\cong\Lambda^2$:
\[
d\alpha\wedge\Phi-\alpha\wedge d\Phi = 3(*\dd_7^{\,8}\alpha - *(\alpha\wedge T_8)_7)-(\alpha\wedge T_{48})_7-( *\dd_{21}^{\,8}\alpha - *(\alpha\wedge T_8)_{21}) -(\alpha\wedge T_{48})_{21}.
\]

For $X\in\Gamma(TM)$, Cartan's formula implies $d(X\lrcorner \Phi) = \mathcal{L}_X\Phi-X\lrcorner d\Phi$. 
From \cref{eq:liePhiexp}
\begin{align*}
    \mathcal{L}_X\Phi = \left(\frac{1}{2}\mathcal{L}_Xg + T(X)+(\nabla(X))_7\right) \diamond \Phi
\end{align*} where $T\in\Omega^1\otimes\Omega^2_7$ is the torsion tensor and 

\[\diamond \ \Phi \colon \Omega^0\oplus Sym^2_0\oplus \Lambda^2_7\oplus \Lambda^2_{21} \to  \Omega^4_1\oplus\Omega^4_7\oplus \Omega^4_{35}
    \] with $\Omega^2_{21}\cong \ker(\diamond \ \Phi)$. 
The $\Omega^4_{1\oplus 35}$ parts in $d(X\lrcorner \Phi)$ are images of symmetric $2$-tensors on $M$ under the $\diamond$-operator. That is, we have
\begin{align*}
    d(X\lrcorner\Phi)& = \left( \left( -\frac{d^*X}{8} + T_8(X)\right) g \right) \diamond \Phi + (T(X)+(\nabla(X))_7+\frac{1}{2}\pi_7(T_8\wedge X))\diamond \Phi \\
    & \quad + \left(Sym^2_0 (\frac{1}{2}\mathcal{L}_Xg + T_8\otimes X)\right)\diamond \Phi -T_8(X)\Phi - X\lrcorner *T_{48}.
\end{align*}
Since from \cref{eq:repdecomp}, $V_8\otimes V_{48}$ has no $V_1$ component $\pi_1(X\lrcorner *T_{48})=0$ and thus we get
\begin{align*}
    \pi_1(d(X\lrcorner \Phi))&= -\frac{1}{2} (d^*X+T_8(X)) \Phi,\\
    \pi_7(d(X\lrcorner \Phi))&= (T(X)+(\nabla(X))_7+\frac{1}{2}\pi_7(T_8\wedge X))\diamond \Phi - \pi_7(X\lrcorner *T_{48}),\\
    \pi_{27}(d(X\lrcorner \Phi))&= - \pi_{27}(X\lrcorner *T_{48}),\\
\pi_{35}(d(X\lrcorner \Phi))&=  \left(Sym^2_0 (\frac{1}{2}\mathcal{L}_Xg + T_8\otimes X)\right)\diamond \Phi - \pi_{35}(X\lrcorner *T_{48}). 
\end{align*} Moreover if we denote by $j_\Phi$ the adjoint of $\diamond\ \Phi$ then by \cite[Proposition 2.4]{Spin7-flow_Dwivedi}

\begin{align*}
\pi_{35}(X\lrcorner *T_{48}) &=\frac12\left(X\lrcorner *T_{48}-T_{48}\wedge X\right),\\
\pi_7(X\lrcorner *T_{48})&=
\frac1{32}j_\Phi\left(\frac12\left(X\lrcorner *T_{48}+T_{48}\wedge X\right)
\right) \diamond \Phi,\\
\pi_{27}(X\lrcorner *T_{48})&=
\frac12\left(X\lrcorner *T_{48}+T_{48}\wedge X^\flat\right)-\pi_7(X\lrcorner *T_{48}).
\end{align*}
Note that since $V_8\otimes V_8$ has no $V_{27}$ component $d^8_{27}(X\lrcorner\Phi)=-\pi_{27}(X\lrcorner *T_{48})$.

%==================== beta in Omega^2 ====================

\medskip

For $\beta \in \Omega^2$

\[
d\beta = (d\beta)_8 + (d\beta)_{48}
        = d\beta_7 + d\beta_{21}
        = \dd_8^{\,7}\beta + \dd_{48}^{\,7}\beta
        + \dd_8^{\,21}\beta + \dd_{48}^{\,21}\beta .
\]
Since 
\[
\beta_7 = \tfrac14\big(*(\beta\wedge\Phi)+\beta\big),\qquad
d\beta_7 = \tfrac14\big(d*(\beta\wedge\Phi)+d\beta\big),
\]
we obtain
\begin{align*}
\dd_8^{\,7}\beta
   &= -\tfrac17 *\!\big(d(\beta_7)\wedge\Phi\big)
    = -\tfrac{1}{28}*\!\big(d*(\beta\wedge\Phi)\wedge\Phi + d\beta\wedge\Phi\big)\\
   &= -\tfrac{1}{28}*\!\big(d(*(\beta\wedge\Phi)\wedge\Phi) + d(\beta\wedge\Phi)-(*(\beta\wedge\Phi)+\beta)\wedge d \Phi\big)\\
   &= -\tfrac17 *\,(d(\beta_7\wedge\Phi)-\beta_7\wedge d\Phi) = \tfrac17\,(3 d^*\beta_7-*(\beta_7\wedge d\Phi)),\\[4pt]
\dd_{48}^{\,7}\beta
   &= d(\beta_7) - \dd_8^{\, 7}\beta\wedge\Phi,\\[4pt]
\dd_8^{\,21}\beta
   &= -\tfrac17 *\!\big(d(\beta_{21})\wedge\Phi\big)
    = -\tfrac17 *\!\big(d(\beta_{21}\wedge\Phi)-\beta_{21}\wedge d\Phi\big)
    = -\tfrac17\,(d^*\beta_{21}+*(\beta_{21}\wedge d\Phi)),\\[4pt]
\dd_{48}^{\,21}\beta
   &= d(\beta)_{21} + \tfrac17 *\!\big(d^*\beta_{21}\wedge\Phi\big).
\end{align*}

Now we compute
$d(\beta\wedge\Phi)\in\Lambda^7\cong\Lambda^1$:
\[
d(\beta\wedge\Phi)
 = d(\beta_7\wedge\Phi + \beta_{21}\wedge\Phi)
 = d(3*\beta_7 - *\beta_{21})
 = 3 d*\beta_7 - d*\beta_{21}.
\]

and
\begin{align*}
d*(\beta\wedge\Phi)
 &= d(3\beta_7 - \beta_{21}) = 3 d\beta_7 - d\beta_{21}\\
 &= \big(3\,\dd_8^{\,7}\beta - \dd_8^{\,21}\beta\big)
  + \big(3\,\dd_{48}^{\,7}\beta - \dd_{48}^{\,21}\beta\big).
\end{align*}

Using Schur's lemma and \cref{eqn:gen_diamond} we have 
\begin{align*}
    d(\beta \diamond \Phi)& = -4 *\dd_8^{\,7}\beta  + 2 *\dd_{48}^{\,7}\beta  -\beta \diamond d\Phi.
\end{align*} 
Thus, one obtains the decomposition of $d\eta$ for $\eta\in\Omega^4_7$ in terms of $\dd^7_8j_\Phi(\eta), \dd^7_{48}j_\Phi(\eta)$. 
%==================== gamma in Omega^3 ====================

\medskip

When $\gamma \in \Omega^3_{48}$

\begin{align*}
    \dd_1^{\,48}\gamma &= \frac{1}{14}*(d\gamma\wedge\Phi)\Phi = \frac{1}{14}*(\gamma \wedge d\Phi)\Phi=\frac{1}{14} g(\gamma ,T_{48}) \Phi, 
\qquad
\dd_{35}^{\,48}\gamma = \frac12(d \gamma - * d \gamma), \\
\dd_7^{\,48}\gamma &= \frac{1}{32} j_\Phi\left(\frac{1}{2}(d\gamma+* d\gamma) -  \dd_1^{\,48}\gamma \right)\diamond \Phi, \qquad \dd_{27}^{\,48}\gamma = \frac{1}{2}(d\gamma +* d\gamma) - \dd_1^{\,48}\gamma-\dd_7^{\,48}\gamma.
\end{align*}
Again since $V_8\otimes V_{48}$ has no $V_1$ component from \cref{eq:repdecomp}, $\pi_1(d(\gamma_{48}))$ is proportional to $T_{48}$. 

\medskip

Since $\gamma$ is $3$-form on an $8$-manifold, using $*d(*\gamma)=-d^*\gamma$, 
\begin{align*}
    \pi^6_7\big(d(*\gamma)\big) &=\frac14\left(d(*\gamma)-\Phi\wedge d^*\gamma\right), \\
    \pi^6_{21}\big(d(*\gamma)\big)&=\frac14\left(3d(*\gamma)+\Phi\wedge d^*\gamma\right).
\end{align*}

\medskip

Now suppose $\eta_{27}\in \Omega^4_{27}$ and $\eta_{35}\in\Omega^4_{35}$. The 8-component of $d\eta_{27}$ comes solely from $T_{48}$ as $V_8 \not\subset V_8\otimes V_{27}$. Since $V_8\otimes V_{27}$ has no $V_8$ component 
$$
X\lrcorner\eta_{27}\in\Lambda^3_{48}
\qquad\text{for every }X\in TM.
$$
There is a canonical Spin(7)-equivariant bilinear contraction

$$
\mathcal C:\Lambda^3_{48}\otimes\Lambda^4_{27}
\longrightarrow\Lambda^1_8,
\qquad
\langle\mathcal C(T_{48},\eta_{27}),X^\flat\rangle
=
\langle T_{48},X\lrcorner\eta_{27}\rangle .
$$

Thus  
\begin{align*}
    \pi_8(d\eta_{27})&=-\frac17\left[*\big(d^*\eta_{27}\wedge \Phi\big)
\right]\wedge \Phi=\frac17 \left( \sum_{j=1}^8 \langle T_{48},e_j\lrcorner\eta_{27}\rangle e^j \right)\wedge\Phi = \frac{1}{7} \ \mathcal C(T_{48},\eta_{27})\wedge\Phi.
\end{align*}
Thus,
 $d\eta_{27} =\pi_8(T_{48}\cdot\eta_{27})+d^{27}_{48}\eta.$ 
Lastly when $\eta_{35}\in\Omega^4_{35}$, $*\eta_{35}=-\eta_{35}$ from \cref{eq:sdasdforms}, and thus 
 \begin{align*}
     \pi_8(d\eta_{35})&= \frac{1}{7}*(d^*\eta_{35} \wedge\Phi)\wedge\Phi
=\frac17*\left[\sum_i(\nabla_{e_i}\eta)\wedge(e_i\lrcorner \Phi)+
\sum_i(e_i\lrcorner \eta)\wedge\big(T(e_i)\diamond \Phi\big)+
\eta\wedge*(T_8\wedge \Phi)-\eta\wedge T_{48}\right]\wedge \Phi.
 \end{align*} 
The 5-form $d\eta_{35}$ has a $V_8$ component coming from the exterior derivative, $T_8$ and $T_{48}$ as evident from the representation theory. 
\end{proof}

We now state the following corollary for torsion-free Spin(7)-structures which we will be using in \Cref{sec:cohomology} for describing the $\ddphi$-Bott--Chern cohomology. This is the analogue of \cite[Table 1]{bryant-someremarks} and \cite[Figure 1]{chan-karigiannis-tsang} for Spin(7)-manifolds.\footnote{The decompositions in \Cref{table:dtorsionfree} have also been derived independently by S. Karigiannis, V. Majewski and T. Pacini (private communication).}

\begin{corollary}
  Let $M^8$ be a smooth manifold with a torsion-free Spin(7)-structure $\Phi$. Then, in the same setting as in \Cref{lem:ddecomp}, we have:

\begin{table}[htp!]
\resizebox{\textwidth}{!}{
 \renewcommand{\arraystretch}{1.2}
\setlength{\arraycolsep}{8pt}
$\begin{array}{l c c c c c c c c} \hline
d f          & = & &   & \dd_8^{\,1} f    & &    & &                 \\ \hline
d(f\Phi)     & = &   &                                & \dd_8^{\,1} f      &  &  & & \\ \hline
d\alpha      & = & & \dd_7^{\,8}\alpha    &                               & \dd_{21}^{\,8}\alpha & & & \\ \hline
d*(\alpha\wedge\Phi) & = & -\frac12\dd_1^{\,8}\alpha  &-\tfrac12 \dd_7^{\,8}\alpha & & & & \dd^8_{35}\alpha & \\ \hline
d(\alpha\wedge\Phi)  & = &   & 3*\dd_7^{\,8}\alpha &           &   -*\dd_{21}^{\,8}\alpha  & &    &             \\ \hline
d(*\alpha)   & = & -\dd_1^{\,8}\alpha    &  &   &     &  &  &            \\ \hline
d\beta_7     & = & &  & \tfrac37*(\dd_8^{\,7}\beta\wedge\Phi)  & & & &\dd_{48}^{\,7}\beta  \\ \hline
d\beta_{21}  & = & & & -\tfrac17*(\dd_8^{\,21}\beta\wedge\Phi) & & & &\dd_{48}^{\,21}\beta \\ \hline
d(\beta_7\wedge\Phi)    & = &  &  & -3*\dd_8^{\,7}\beta & & & &      \\ \hline
d(\beta_{21}\wedge\Phi) & = &    &    & *\dd_8^{\,21}\beta         &   &       &  &               \\ \hline
d\gamma_{48} & = & & \dd_7^{\,48}\gamma  &      &       &    \dd_{27}^{\,48}\gamma    &\dd_{35}^{\,48}\gamma    &                     \\ \hline
d*\gamma_{48}& = & &\tfrac1{12}(j_\Phi(\dd_7^{\,48}\gamma)\wedge\Phi) &  & -(\dd_{21}^{\,48}\gamma\wedge\Phi) &   &      &      \\ \hline
d\eta_{27}   & = &  & &  &  &  &  & \dd_{48}^{\,27}\eta  \\ \hline
d\eta_{35}   & = & &  & \dd_8^{\,35}\eta & &  &  & \dd_{48}^{\,35}\eta  \\ \hline
\end{array}$}
\caption{Exterior derivative formulae for torsion-free $\Phi$}
\label{table:dtorsionfree}
\end{table}  
\end{corollary}

%\newpage
\begin{proof}
The proof immediately follows from \Cref{table:dwithtorsion}. When $\Phi$ is torsion free $\dd^p_{q\to q'}\neq0$ iff $V_{q'}$ occurs in both $V_8\otimes V_q$ and $\Lambda^{p+1}$.
\end{proof}

\subsection{Second-order identities and Laplacians for torsion-free Spin(7)-structures}

The identity $d^2 = 0$ implies that
\begin{align*}
\sum_{q}\dd^{q}_r\dd^p_q \omega=0,\ \qquad \omega\in \Omega^k(M),
\end{align*}
and the sum is over the irreducible summands of $\Omega^{k+1}(M)$. This is equivalent to the second order identities on the operators $\dd^p_q$
listed in \Cref{table:d2}. Finally, the formulas for the Hodge Laplacians in terms of the operators $\dd^p_q$ are as given in \Cref{table:d3}. The bracketed numbers on the right-hand side of \Cref{table:d3} indicate the normalized principal-symbol weights of the summands in the same order.

We assume throughout this section that \(\nabla\Phi=0\). Hence all Spin(7)-type
projections are parallel. Since the decomposition of $\Omega^k=\bigoplus_{p}\Omega^k_p$ into Spin(7)-irreducible representations is orthogonal, we have

\begin{align*}
    d^*d&=\sum_q (\dd^p_q)^*\dd^p_q,\qquad dd^*=\sum_r \dd^r_p(\dd^r_p)^*.
\end{align*}

\noindent
Therefore

\begin{align}\label{eqn:laplacian_decomp}
    \Delta|_{\Omega^k_p}&=d^*d+dd^*=\sum_q (\dd^p_q)^*\dd^p_q+\sum_r \dd^r_p(\dd^r_p)^*.
\end{align}

\begin{table}[htp!]
\centering
\[
\begin{aligned}
\Omega^0:\ & \ddtwo{8}{7}\ddtwo{1}{8}=0,\qquad \ddtwo{8}{21}\ddtwo{1}{8}=0.\\[2pt]
\Omega^1_8:\ & \ddtwo{7}{8}\ddtwo{8}{7}+\ddtwo{21}{8}\ddtwo{8}{21}=0,\qquad
              \ddtwo{7}{48}\ddtwo{8}{7}+\ddtwo{21}{48}\ddtwo{8}{21}=0.\\[2pt]
\Omega^2_7:\ & \ddtwo{8}{1}\ddtwo{7}{8}=0,\qquad
               \ddtwo{8}{7}\ddtwo{7}{8}+\ddtwo{48}{7}\ddtwo{7}{48}=0,\\
             & \ddtwo{48}{27}\ddtwo{7}{48}=0,\qquad
               \ddtwo{8}{35}\ddtwo{7}{8}+\ddtwo{48}{35}\ddtwo{7}{48}=0.\\[2pt]
\Omega^2_{21}:\ & \ddtwo{8}{1}\ddtwo{21}{8}=0,\qquad
                  \ddtwo{8}{7}\ddtwo{21}{8}+\ddtwo{48}{7}\ddtwo{21}{48}=0,\\
                & \ddtwo{48}{27}\ddtwo{21}{48}=0,\qquad
                  \ddtwo{8}{35}\ddtwo{21}{8}+\ddtwo{48}{35}\ddtwo{21}{48}=0.\\[2pt]
\Omega^3_8:\ & \ddtwo{1}{8}\ddtwo{8}{1}+\ddtwo{7}{8}\ddtwo{8}{7}+\ddtwo{35}{8}\ddtwo{8}{35}=0,\qquad
               \ddtwo{7}{48}\ddtwo{8}{7}+\ddtwo{35}{48}\ddtwo{8}{35}=0.\\[2pt]
\Omega^3_{48}:\ & \ddtwo{7}{8}\ddtwo{48}{7}+\ddtwo{35}{8}\ddtwo{48}{35}=0,\qquad
                  \ddtwo{7}{48}\ddtwo{48}{7}+\ddtwo{27}{48}\ddtwo{48}{27}+\ddtwo{35}{48}\ddtwo{48}{35}=0.
\end{aligned}
\]
\caption{Second-order identities $d^2=0$}
\label{table:d2}
\end{table}

\begin{table}[htp!]
\centering
\[
\begin{aligned}
\Delta f&=\ddtwos{1}{8}\ddtwo{1}{8}\,f && [\,1\,]=\nabla^\ast\nabla f\\[3pt]
\Delta\alpha&=\big[\ddtwos{8}{7}\ddtwo{8}{7}+\ddtwos{8}{21}\ddtwo{8}{21}+\ddtwo{1}{8}\ddtwos{1}{8}\big]\alpha
&&\big[\tfrac{7}{32},\tfrac{21}{32},\tfrac{4}{32}\big]=\nabla^\ast\nabla\alpha\\[3pt]
\Delta\gamma&=\big[\ddtwos{7}{8}\ddtwo{7}{8}+\ddtwos{7}{48}\ddtwo{7}{48}+\ddtwo{8}{7}\ddtwos{8}{7}\big]\gamma
&&\big[\tfrac{9}{28},\tfrac{12}{28},\tfrac{7}{28}\big]=\nabla^\ast\nabla\gamma\\[3pt]
\Delta\beta&=\big[\ddtwos{21}{8}\ddtwo{21}{8}+\ddtwos{21}{48}\ddtwo{21}{48}+\ddtwo{8}{21}\ddtwos{8}{21}\big]\beta
&&\big[\tfrac{1}{28},\tfrac{20}{28},\tfrac{7}{28}\big]=\nabla^\ast\nabla\beta-2\R|_{\sff}\beta\\[3pt]
\Delta\tau&=\big[\ddtwos{48}{7}\ddtwo{48}{7}+\ddtwos{48}{27}\ddtwo{48}{27}+\ddtwos{48}{35}\ddtwo{48}{35}
             +\ddtwo{7}{48}\ddtwos{7}{48}+\ddtwo{21}{48}\ddtwos{21}{48}\big]\tau
&&\big[\tfrac{1}{32},\tfrac{9}{32},\tfrac{10}{32},\tfrac{2}{32},\tfrac{10}{32}\big]\\[3pt]
\Delta\rho&=\big[\ddtwos{27}{48}\ddtwo{27}{48}+\ddtwo{48}{27}\ddtwos{48}{27}\big]\rho
&&\big[\tfrac12,\tfrac12\big]\\[3pt]
\Delta\lambda&=\big[\ddtwos{35}{8}\ddtwo{35}{8}+\ddtwos{35}{48}\ddtwo{35}{48}
               +\ddtwo{8}{35}\ddtwos{8}{35}+\ddtwo{48}{35}\ddtwos{48}{35}\big]\lambda
&&\big[\tfrac{1}{14},\tfrac{6}{14},\tfrac{1}{14},\tfrac{6}{14}\big]
\end{aligned}
\]
\caption{The Hodge Laplacian with symbol weights, $f\in \Omega^0, \alpha\in\Omega^1_8,\ \gamma\in\Omega^2_7,\ \beta\in\Omega^2_{21},\ \tau\in\Omega^3_{48},\ \rho\in\Omega^4_{27},\ \lambda\in\Omega^4_{35}$}
\label{table:d3} 
\end{table}

\noindent
Moreover, the Weitzenböck formula gives
\begin{align*}
    \Delta=\nabla^*\nabla+q(R).
\end{align*}
For a torsion-free Spin(7)-structure the metric is Ricci-flat. Hence
$q(R)=0$ on functions and $1$-forms. Moreover the curvature operator
$R:\Lambda^2\to\Lambda^2$ has image in
$\Lambda^2_{21}\cong\mathfrak{spin}(7)$, and annihilates
$\Lambda^2_7$. Thus
\begin{align*}
    \Delta|_{\Omega^0_1}&=\nabla^*\nabla,
\qquad
\Delta|_{\Omega^1_8}=\nabla^*\nabla,
\qquad
\Delta|_{\Omega^2_7}=\nabla^*\nabla,
\end{align*}
while on \(\Omega^2_{21}\),
\begin{align*}
    \Delta=\nabla^*\nabla-2R|_{\mathfrak{spin}(7)},
\end{align*}and on the remaining types $\Delta=\nabla^\ast\nabla+\mathcal C$ with
$\mathcal C\in S^2(\sff)$ built from $\R$.

For a first-order summand $P$ appearing in \cref{eqn:laplacian_decomp}, the normalized symbol
weight on $E=\Omega^k_p$ is given by by
\begin{align*}
    \overline c(P)&=\frac{1}{\dim E}
\operatorname{tr}_E\left(|\xi|^{-2}\sigma_2(P)(\xi)\right),\qquad |\xi|=1.
\end{align*}
Since Spin(7) acts transitively on the unit sphere in $\mathbb R^8$, $\overline c(P)$ is independent of the choice of the unit vector $\xi$. Let $E=\Lambda^k_p$ be one irreducible Spin(7)-summand, and fix a unit vector $\xi$. Define
\begin{align*}
  \varepsilon_\xi(\omega)&=\xi\wedge \omega,
\qquad
\iota_\xi(\omega)=\xi^\sharp\lrcorner \omega.  
\end{align*}
Then the principal symbol of $(d^p_q)^*d^p_q$ on $E=\Lambda^k_p$ is $\iota_\xi\pi_q\varepsilon_\xi:E\to E$. The normalized symbol weight of this term is
\begin{align*}
    \overline c(d^p_q)
=
\frac1{\dim E}
\operatorname{tr}_E\big(\iota_\xi\pi_q\varepsilon_\xi\big).
\end{align*}
Similarly, for the term $d^r_p(d^r_p)^*,$ the symbol weight is
\begin{align*}
    \overline c(d^r_p)^*=\frac1{\dim E}\operatorname{tr}_E\big(\varepsilon_\xi\pi_r\iota_\xi\big).
\end{align*}
For $E\subset \Lambda^k$, since $|\varepsilon_\xi\omega|^2+|\iota_\xi\omega|^2=|\omega|^2$,
over an orthonormal basis $\{e^i\}$, using
\begin{align*}
    \sum_i \varepsilon_{e^i}\iota_{e^i}&=k\,\mathrm{id}_{\Lambda^k},
\qquad
\sum_i \iota_{e^i}\varepsilon_{e^i}=(8-k)\,\mathrm{id}_{\Lambda^k},
\end{align*} gives
\begin{align}\label{eqn:sum_c}
    \frac1{\dim E}\operatorname{tr}_E(\varepsilon_\xi\iota_\xi)&=\frac{k}{8}, \qquad
    \frac1{\dim E}\operatorname{tr}_E(\iota_\xi\varepsilon_\xi)=\frac{8-k}{8},
\end{align} which sums to $1$ as required by $\sigma_2(\Delta)(\xi)=|\xi|^2\mathrm{id}$.  The computations of the normalized symbol weights for the summand in \Cref{table:d3} are straightforward but cumbersome. We demonstrate it for $\tau\in\Omega^3_{48}$ and note that one can use similar techniques to compute all the symbol weights. 

Take $\Phi$ to be the standard Cayley form and $\xi=e^1$. At any point $p\in M$ define $W=\langle e^1\rangle^\perp$. We write 
\begin{align*}
    \Phi=e^1\wedge\varphi + \psi,
\end{align*} where $\varphi \in \Omega^3(W)$ is a $\G2$-structure on $W$ and $\psi=*_W\varphi$. Since $\rm{G}_2\subset \operatorname{Spin(7)}$ we have the $\rm{G}_2$-orthogonal decomposition 
\begin{align*}
    \Omega^3_{48}&= \Omega^3_{27}\oplus e^1\wedge\Omega^2_{14}\oplus\left\{4e^1\wedge (u\lrcorner \varphi)+3(u\lrcorner \psi):u\in W\right\}.
\end{align*} Under the above decomposition if  $\tau = \xi_1+e^1\wedge\xi_2 +4e^1\wedge (u\lrcorner \varphi)+3(u\lrcorner \psi) $ then $\iota_{e^1}\tau = \beta+4 (u\lrcorner \varphi)$. Since $\beta\in\Lambda^2_{14}$, 
\begin{align*}
    \pi_7(\iota_{e_1}\tau) &= 4\pi_7(u\lrcorner \varphi) = 3(u\lrcorner\varphi+e^1\wedge u). 
\end{align*} Combining the three ${\rm{G}}_2$-summand of $\Lambda^3_{48}$ and using the identities $\|u\lrcorner\varphi\|^2=3\|u\|^2, \|u\lrcorner\psi\|^2=4\|u\|^2$, we get
\begin{align*}
   \overline{c}((\dd^{7}_{48})^*)&= \frac{1}{48}\operatorname{tr}_{\Lambda^3_{48}}
|\pi_7\iota_{e_1}\tau|^2=\frac{7}{48}\cdot\frac{9\|u\lrcorner\varphi+e^1\wedge u\|^2}{\|4e^1\wedge (u\lrcorner \varphi)+3(u\lrcorner \psi)\|^2} = \frac{7}{48}\cdot \frac{3}{7} =\frac{1}{16}.
\end{align*}
Similarly since $ \pi_{21}(\iota_{e_1}\tau) =  \beta + u\lrcorner \varphi -3 e^1\wedge u$, 
\begin{align*}
   \overline{c}((\dd^{21}_{48})^*)&= \frac{1}{48}\operatorname{tr}_{\Lambda^3_{48}}
|\pi_{21}\iota_{e_1}\tau|^2=\frac{14}{48} \|\xi_2\|^2+\frac{7}{48}\cdot\frac{\|u\lrcorner\varphi-3e^1\wedge u\|^2}{\|4e^1\wedge (u\lrcorner \varphi)+3(u\lrcorner \psi)\|^2}  =\frac{5}{16}.
\end{align*}
The identity $\pi^2_7(\iota_{e_1}\tau)\diamond \Phi=-8\,\pi^4_7(e^1\wedge\tau)$, and $|A\diamond\Phi|^2=32|A|^2$ imply
\begin{align*}
    \overline{c}(\dd^{48}_{7})&=\frac{1}{2} \overline{c}((\dd^{7}_{48})^*) =\frac{1}{32}. 
\end{align*}
Similarly using that $\pi_{35}(e^1\wedge\tau)$ is the anti-self dual part of $e^1\wedge\tau$ we compute 
\begin{align*}
    \overline{c}(\dd^{48}_{35})&=\frac{1}{48}\left(\frac{27\|e^1\wedge\xi_1\|^2}{2}+\frac{7}{2} \frac{9\|e^1\wedge(u\lrcorner\psi)\|^2}{\|4e^1\wedge (u\lrcorner \varphi)+3(u\lrcorner \psi)\|^2}\right)= \frac{5}{16},
\end{align*}
and then \cref{eqn:sum_c} implies
\begin{align*}
    \overline{c}(\dd^{48}_{27})=\frac{5}{8}-\overline{c}(\dd^{48}_{7})-\overline{c}(\dd^{48}_{35}) = \frac{9}{32}.
\end{align*}

\section{Bott-Chern type Cohomology of Spin(7)-manifolds}\label{sec:cohomology}

In this section, we study $\Ima(\ddphi)$ in detail which  allows us to find the subspaces of $\Omega^4(M)$ in which to set-up the Hodge decomposition theorem. Throughout this section, $(M^8, \Phi)$ is a Spin(7)-manifold and hence $d\Phi=0$. We will prove our Hodge decomposition result in terms of $\Ima(\ddphi)$ in \Cref{thm:Spin7Hodge}. Using this, we prove the analogue in Spin(7)-geometry of the $\pt\bar{\pt}$-lemma in K\"ahler geometry. Finally, we give some applications of our $\ddphi$-lemma to moduli space of torsion-free Spin(7)-structures and calibrated geometry.  

\medskip

%\subsection{Bott--Chern type cohomology  and the $\ddphi$-lemma}
Given $f\in \Omega^0(M)$, since $\ddphi(f) = d(df\lrcorner \Phi)$, trivially $\Ima(\ddphi)\subset \Ima(d)$. Moreover, from \Cref{table:dtorsionfree} or using the Lie derivative formula \cref{eq:liePhiexp}, we see that
\begin{align}\label{eq:imddphi1}
d(df\lrcorner \Phi) = \cL_{df}\Phi = \frac 12 \cL_{df}g \diamond \Phi= -\frac 12 d^*(df)\Phi + \left(\frac 12 \cL_{df}g \right)_{0}\diamond \Phi = -\frac 12 (\Delta f)\Phi + \left(\frac 12 \cL_{df}g \right)_{0}\diamond \Phi,
\end{align}
where $\left(\frac 12 \cL_{df}g \right)_{0}$ denotes the trace-free part of the symmetric $2$-tensor $\cL_{df}g$. In other words,
\begin{align}
\ddphi f\in \Omega^4_1(M)\oplus \Omega^4_{35}(M). \label{eq:imddphi2}
\end{align}

In fact, \cref{eq:imddphi1} shows that if $M$ is a compact Spin(7)-manifold then by the maximum principle, $\Delta f=0 \implies f=\text{constant} \implies df=0$ and hence
\begin{align}
\text{for\ compact\ Spin(7)-manifolds}\ \ \ \ \pi^4_1(\ddphi f) =0 \iff \pi^4_{35}(\ddphi f)=0. 
\end{align}

We have the following proposition.

\begin{proposition}\label{prop:pi348d*zero}
Let $(M^8, \Phi)$ be a Spin(7)-manifold. If $\gamma\in \Omega^4_1(M)\oplus \Omega^4_{35}(M)$ such that $d\gamma=0$, then $d^*\gamma \in \Omega^3_8(M)$, that is, $\pi^3_{48}(d^*\gamma)=0$. In particular, $\pi^3_{48}d^*(\ddphi f)=0$ for $f\in C^{\infty}(M)$.
\end{proposition}

\begin{proof}
Let $\gamma=h\Phi + \eta \in \Omega^4_{1\oplus 35}(M)$ with $h\in C^{\infty}(M)$ and $\eta\in \Omega^4_{35}$. Since $d\gamma=0$ and the exterior derivative is linear, we get
\begin{align*}
dh\wedge \Phi = -d\eta \implies d\eta \in \Omega^5_8(M).
\end{align*}
Using \cref{eq:sdasdforms}, we have 
\begin{align*}
d^*\eta =-*d*\eta=*d\eta = -dh\lrcorner \Phi \in \Omega^3_8(M).   
\end{align*}
Moreover, $d^*(h\Phi)=-dh\lrcorner \Phi \in \Omega^3_8(M)$. Thus, $d^*\gamma \in \Omega^3_8(M)$ and hence $\pi^3_{48}(d^*\gamma)=0$. The assertion about $\ddphi f$ follows as $\ddphi f\in \Omega^4_{1\oplus 35}(M)$ and $d(\ddphi f)=0$.
\end{proof}

\begin{remark}
The proof above shows that the result still holds if we just assume $T_{48}=0$ instead of the  stronger assumption of $d\Phi=0$. However, the proposition is not true for arbitrary Spin(7)-structures. \demo
\end{remark}

An upshot of \cref{eq:imddphi2} and the previous proposition is that $\Ima(\ddphi)\subset \Omega^4_{1\oplus 35}$, $\Ima(\ddphi)\subset \ker(\pi^3_{48}d^*)$ and $\Ima(\ddphi)\subset \ker(d)$. Also, for $\gamma \in \Ima(\ddphi)$ we have $d\gamma=0 \iff d^*\gamma \in \Omega^3_8(M)$. In other words, in the space $\Omega^4(M)$, we have
\begin{align}\label{eq:imddphichar}
\Ima(\ddphi)\subset \left(\Omega^4_1(M)\oplus \Omega^4_{35}(M) \right)\cap \ker(\pi^3_{48}d^*)\cap \ker(d)=\left(\Omega^4_1(M)\oplus \Omega^4_{35}(M) \right) \cap \ker(d).
\end{align}
We want to understand the complement of the space $\Ima(\ddphi)$ in $\Omega^4(M)$. For this, we would need to understand the adjoint of the operator $\ddphi$. For analyzing the space $\Omega^4_{1\oplus35}$ in terms of the operator $\ddphi$ and its adjoint, we need the following simple proposition.

\begin{proposition}\label{prop:psymbold}
The principal symbol $\sigma(\ddphi)$ of the operator $\ddphi$ is injective.
\end{proposition}

\begin{proof}
Since $\ddphi:\Omega^0(M)\rightarrow \Omega^4_1(M)\oplus \Omega^4_{35}(M)$ is a linear differential operator, hence for $x\in M$ and non-zero covector $\xi\in T^*_xM$, the principal symbol is a homomorphism
\begin{align*}
\sigma(\ddphi)_x(\xi): \bR\rightarrow \Omega^4_1(T^*_xM)\oplus \Omega^4_{35}(T^*_xM).
\end{align*}
Writing the $\ddphi$-operator in terms of the Levi-Civita connection $\del$ and using the Fourier convention of replacing the covariant derivative $\del_j$ by $\xi_j$, we have
\begin{align*}
\left(\sigma(\ddphi)_x(\xi)(c)\right)_{ijkl}&=c(\xi_i\xi^p\Phi_{pjkl}-\xi_j\xi^p\Phi_{pikl}+\xi_k\xi^p\Phi_{pijl}-\xi_l\xi^p\Phi_{pijk}),
\end{align*}
which on contracting both sides by $\Phi^{ijkl}$ and using \cref{eq:impiden4} gives
\begin{align*}
\left(\sigma(\ddphi)_x(\xi)(c)\right)_{ijkl}\Phi^{ijkl}&=c(\xi_i\xi^p\Phi_{pjkl}-\xi_j\xi^p\Phi_{pikl}+\xi_k\xi^p\Phi_{pijl}-\xi_l\xi^p\Phi_{pijk})\Phi^{ijkl}\\
&=168c|\xi|^2.
\end{align*}
Since $\xi\neq 0$, the symbol $\sigma(\ddphi)$ is injective.
\end{proof}

Thus, \Cref{prop:pi348d*zero} and \Cref{prop:psymbold} tells us that the right space to study the cohomological properties of a Spin(7)-manifold associated with the $\ddphi$-operator is the space
\begin{align}\label{eq:cohospace}
K= \left(\Omega^4_1(M)\oplus \Omega^4_{35}(M) \right)\cap \ker(d).
\end{align}

\begin{remark}
A comparison with the analogue of the space $K$ in the Calabi--Yau or the torsion-free $\G2$-case in \cite[Def. 6.10, Def. 5.10]{pacini-raferro-pluripotential} shows  that the correct space in which to set-up the $\ddphi$-cohomology is simpler in the Spin(7)-case and is more akin to Bott--Chern cohomology of K\"ahler manifolds. This happens because the space $\Omega^4_{35}(M)$ is the space of anti-self-dual $4$-forms on a manifold with a Spin(7)-structure.
\end{remark}

%In fact, we can drop the $\ker(\pi^3_{48}d^*)$ term from $K$ without any change in the discussion, see \Cref{rem:differentK}. 

Recall, for instance, from \cite[Appendix]{pacini-raferro-pluripotential} that if $P:\Gamma(E)\rightarrow \Gamma(F)$ is a differential operator of order $m$ between sections of vector bundles $E$ and $F$ and the principal symbol of $P$ is injective then on a compact manifold we have that $P:H^m(E)\rightarrow L^2(F)$ has closed image and there is an $L^2$-orthogonal decomposition
\begin{align*}
L^2(F) = \Ima(P)\oplus \ker (P^*),
\end{align*}
where $P^*$ is the Hilbert-space adjoint of the unbounded operator $P: L^2(E)\rightarrow L^2(F)$. Moreover, the abstract adjoint $P^*$ agrees with the distributional extension of the formal adjoint $P^t$. Also, $H^m(E)$ is the Sobolev space of sections of $E$ with $m$ weak derivatives in $L^2$.

Since we are interested in the $\ddphi$-operator which is an operator of order $2$ and whose principal symbol is injective, we have from general theory of such operators that
\begin{align*}
L^2(\Omega^4_1(M)\oplus \Omega^4_{35}(M))= \Ima(\ddphi) \oplus \ker ((\ddphi)^*),
\end{align*}
where $(\ddphi)^*$ is the Hilbert-space adjoint of the $\ddphi$-operator. We denote the formal adjoint of the $\ddphi$-operator on the space $\Omega^4_{1\oplus 35}(M)$ by $(\ddphi)^t$. Let us compute that.

Let $h\Phi+\eta \in \Omega^4_1(M)\oplus \Omega^4_{35}(M)$. Then
\begin{align*}
\int_M \left\langle \ddphi f, h\Phi + \eta  \right\rangle \vol_{\Phi}&=\int_M \left\langle df\lrcorner \Phi, d^*(h\Phi+\eta)  \right\rangle \vol_{\Phi}\\
&= \int_M \left\langle *(df\wedge \Phi), d^*(h\Phi +\eta) \right\rangle \vol_{\Phi}\\
&=\int_M  (df\wedge \Phi) \wedge d^*(h\Phi+\eta)  \\
&=\int_M d(f\Phi\wedge d^*(h\Phi+\eta))-\int_M f\Phi \wedge dd^*(h\Phi+\eta)\\
&=-\int_M f\Phi \wedge dd^*(h\Phi+\eta).
\end{align*}
We used the facts that $d\Phi=0$ and Stokes's theorem in the above computations. Thus, we get
\begin{align}\label{ddphiadjoint}
(\ddphi)^t: \Omega^4_1(M)\oplus \Omega^4_{35}(M)\rightarrow \Omega^0(M)\ \ \text{is}\ (\ddphi)^t(h\Phi+\eta) = -*(dd^*(h\Phi+\eta)\wedge \Phi).
\end{align}

We now state and prove our main result on the decomposition of the space $K$ in \cref{eq:cohospace}. The proof of the next theorem is inspired by the analogous result of Pacini--Raffero \cite[Thm. 5.12]{pacini-raferro-pluripotential} for $\G2$-manifolds.

\begin{theorem}\label{thm:Spin7Hodge}
Let $(M^8, \Phi)$ be a compact manifold with a torsion-free Spin(7)-structure $\Phi$. Then there exists an $L^2$-orthogonal decomposition of the space $K$ in \cref{eq:cohospace} as
\begin{align}\label{eq:Spin7Hodge}
K=\Ima(\ddphi) \oplus \cH^4_1\oplus \cH^4_{35},
\end{align}
where $\cH^4_1$ and $\cH^4_{35}$ denote the space of harmonic forms in the spaces $\Omega^4_1$ and $\Omega^4_{35}$ respectively.
\end{theorem}

\begin{proof}
Since the operator $\ddphi$ has an injective symbol from \Cref{prop:psymbold}, we use the general theory about such operators to find an $L^2$-orthogonal decomposition of the space $K\supseteq \Ima(\ddphi)$. Consider the Hilbert space
\begin{align*}
\cA = L^2(\Omega^4_1(M)\oplus \Omega^4_{35}(M)) \cap \ker d,
\end{align*}
where $d$ is intrepreted distributionally. As we saw in the discussions leading up to \cref{eq:cohospace} that $\Ima(\ddphi)\subset \cA$, hence we may regard
\begin{align*}
\ddphi: L^2(M) \rightarrow \cA,
\end{align*}
with domain, the Sobolev space $H^2(M)$. Thus,
\begin{align}\label{eq:hilbertspaceimage}
\cA = \Ima(\ddphi) \oplus \ker \left((dd^{\Phi})^*\mid_{\cA}\right).
\end{align}

As we are interested in the space $K$, we look at the auxiliary operator 
\begin{align*}
A&: \Omega^4_1\oplus \Omega^4_{35} \rightarrow \Omega^0 \oplus \Omega^5\\
& \alpha \mapsto \left((\ddphi)^t\alpha,\ d\alpha \right).
\end{align*}
Thus, 
\begin{align*}
\ker A =\{ \gamma \in \Omega^4_1\oplus \Omega^4_{35}\ :\ (\ddphi)^t\gamma=0, d\gamma=0\} = \ker\left( (\ddphi)^t\mid_K \right).    \end{align*}
The reason we introduce the auxiliary operator $A$ is to  guarantee smoothness of the $L^2$-forms which will be obtained by proving the injectivity of the principal symbol of the operator $A$.

We prove that the principal symbol of the operator $A$ is injective. Let $p\in M$ and suppose $\xi\in T^*_pM$ be a non-zero covector. We use the same notation $\xi$ for indentifying $\xi$ as both a vector and a covector. Using \cref{ddphiadjoint}, the principal symbol $$\sigma(A)_p(\xi)(\alpha)=\left(-*(\xi\wedge(\xi\lrcorner \alpha)\wedge \Phi), \xi\wedge \alpha \right).$$If $\alpha\in \ker(\sigma(A)_p(\xi))\subset \Omega^4_{1\oplus 35}(T^*_pM)$, then using $\xi\wedge \alpha=0$ we have
\begin{align*}
0=\xi\lrcorner(\xi \wedge \alpha) = |\xi|^2\alpha - \xi\wedge(\xi\lrcorner \alpha),
\end{align*}
which on using the fact that $*(\xi\wedge(\xi\lrcorner \alpha)\wedge \Phi)=0$ gives
\begin{align*}
0&= |\xi|^2*(\alpha \wedge \Phi),
\end{align*}
which on using the fact that for $a\Phi\in \Omega^4_1(M)$, $a\Phi\wedge \Phi=14a\vol_{\Phi}=0$ implies $a=0$, shows $\alpha \in \Omega^4_{35}(M)$. Since $\Omega^4_{-}=\Omega^4_{35}$, we have that $\alpha$ is anti-self-dual, $*\alpha=-\alpha$. Thus, again using $\xi\wedge \alpha=0$ (as $\alpha \in \ker(\sigma(A)_p(\xi)))$ implies
\begin{align*}
0=*(\xi\wedge \alpha)=\xi\lrcorner *\alpha = -\xi \lrcorner \alpha,
\end{align*}
thus proving as before,
\begin{align*}
|\xi|^2\alpha=0\ \implies\ \alpha=0.
\end{align*}
Hence the principal symbol of the auxiliary operator $A$ is injective. By standard analytic theory, $\ker A$ consists of smooth forms and as a result
\begin{align*}
\ker \left((\ddphi)^*\mid_{K} \right)=\ker \left((\ddphi)^t\mid_{K} \right).
\end{align*}

\medskip

\noindent
\begin{claim}\label{claim:Spin7Hodge}
$\ker \left((\ddphi)^t\mid_{K} \right) = \cH^4_1\oplus\cH^4_{35}$.
\end{claim}
If $\eta\in \cH^4_1\oplus\cH^4_{35}$ then in particular, $d^*\eta =0$ and hence $\eta\in \ker \left((\ddphi)^t\mid_{K} \right)$. For the converse, notice that since $K\subset \ker(d)$, hence from \cref{ddphiadjoint}
\begin{align*}
(\ddphi)^t\mid_K = -*((dd^*+d^*d)(\cdot)\wedge \Phi)=-*(\Delta(\cdot) \wedge \Phi),
\end{align*}
and thus
\begin{align*}
\ker \left((\ddphi)^t\mid_{K} \right)=\ker(\Delta(\cdot)\mid_{\Omega^4_{1\oplus 35}}\wedge \Phi)\cap K.
\end{align*}
Since $\Phi$ is torsion-free hence $\Delta(\Omega^4_{i})\subset \Omega^4_i,\ i=1,\ 35$. Let $\gamma=\gamma_1+\gamma_{35}\in \ker(\Delta(\cdot)\mid_{\Omega^4_{1\oplus 35}}\wedge \Phi)\cap K$. If $\gamma_1=a\Phi$ then $\Delta(a\Phi)\wedge \Phi=14(\Delta a)\vol_{\Phi}$ and hence $\gamma_1\in \ker(\Delta(\cdot)\mid_{\Omega^4_{1\oplus 35}}\wedge \Phi)\cap K \implies\ a=\text{constant}$ and thus
\begin{align*}
\ker(\Delta(\cdot)\mid_{\Omega^4_{1\oplus 35}}\wedge \Phi)\cap K = \bR\cdot \Phi \oplus \left(\Omega^4_{35}\cap K \right).
\end{align*}
Again, when $\gamma_{35}\in K \implies d\gamma_{35}=0$ and since $\Omega^4_{35}=\Omega^4_{-}$ hence $d^*\gamma_{35}=0$ which implies that $\gamma_{35}\in \cH^4_{35}$. Thus, $\ker(\Delta(\cdot)\mid_{\Omega^4_{1\oplus 35}}\wedge \Phi)\cap K=\cH^4_1\oplus \cH^4_{35}=\ker((\ddphi)^t\mid_K)$ which proves the claim.

\medskip

Thus, from \cref{eq:hilbertspaceimage} and \Cref{claim:Spin7Hodge}, we have
\begin{align*}
\cA = \Ima(\ddphi) \oplus \cH^4_1\oplus \cH^4_{35}.
\end{align*}
This identification is in the distributional sense and to get the equality for smooth forms, we intersect both sides of the equation with the space $K$ to get
\begin{align*}
K=\Ima(\ddphi) \oplus \cH^4_1\oplus \cH^4_{35},
\end{align*}
which proves \cref{eq:Spin7Hodge}.
\end{proof}

\begin{remark}
In the proof above, we explicitly used the fact that the Laplacian $\Delta$ preserves the type of form which is only true for torsion-free Spin(7)-structures. Any analogue of \Cref{thm:Spin7Hodge} for Spin(7)-structures with torsion will need to incorporate torsion terms as well. \demo    
\end{remark}

\begin{remark}\label{rem:differentK}
As is evident from the proof above, we didn't need to use the fact that $\pi^3_{48}d^*(\ddphi)=0$ either in proving the injectivity of the principal symbol of the auxiliary operator $A$ or in proving \Cref{claim:Spin7Hodge}. Both these assertions simply follow from the fact that $\Omega^4_{35}=\Omega^4_-$ on a manifold with a Spin(7)-structure. This is in contrast to the $\G2$-case as in \cite[Thm. 512]{pacini-raferro-pluripotential}. \demo
\end{remark}

As a corollary of \Cref{thm:Spin7Hodge} we get the analogue of the $\pt\bar{\pt}$-lemma in K\"ahler geometry.

\begin{lemma}[\textbf{Global $\ddphi$-lemma}]\label{lem:Spin7delbarlemma}
Let $(M^8, \Phi)$ be a compact Spin(7)-manifold and let $\gamma \in \Omega^4_1\oplus \Omega^4_{35}$. If $\gamma$ is globally $d$-exact, then $\gamma$ is globally $\ddphi$-exact, that is, $\gamma=\ddphi f$ for some $f\in C^{\infty}(M)$. 
\end{lemma}

\begin{proof}
As $\del \Phi=0$ hence the Laplacian $\Delta$ preserves the type decompositions on $\Omega^4$ and thus, $\cH^4(M)\cap (\Omega^4_1(M)\oplus \Omega^4_{35}(M))=\cH^4_1(M)\oplus \cH^4_{35}(M)$. Since $\gamma \in \Ima(d)\subset \ker(d)$ hence it is closed and thus $\gamma \in K$. It follows from \cref{eq:Spin7Hodge} that $\gamma \in \Ima(\ddphi) \oplus \cH^4_1\oplus \cH^4_{35}$. But since $\Ima(d)$ is orthogonal to the space of harmonic forms, we get that $\gamma \in \Ima(\ddphi)$.
\end{proof}

\begin{remark}
The $\ddphi$-lemma for Spin(7)-manifolds uses the fact that the Spin(7)-structure is torsion-free. For using this type of lemma as in K\"ahler geometry, we need a version of a $\ddphi$-lemma for \emph{arbitrary} Spin(7)-structures which should incorporate the torsion of $\Phi$. This is an interesting question for future. \demo
\end{remark}

Proceeding in the same way as in the case of complex manifolds \cite{bott-chern} or in the case of $\G2$-manifolds \cite[\textsection5.5]{pacini-raferro-pluripotential}, we can define the analogue of Bott--Chern cohomology for Spin(7)-manifolds.

\medskip

\begin{definition}[\textbf{Bott--Chern-type cohomology}]\label{def:bot-chernspin7}
Let $(M^8, \Phi)$ be a compact Spin(7)-manifold. Then the sequence 
\begin{align*}
\Omega^0(M) \xrightarrow{\ddphi} \Omega^4_1(M) \oplus \Omega^4_{35}(M) \xrightarrow{d} \Omega^5(M)
\end{align*}
defines a complex and hence we define the \textbf{Bott--Chern-type $\ddphi$-cohomology space} as
\begin{align}\label{eq:bot-cherncoh}
H^{\Phi}(M) = \cfrac{\ker (\left.d\right|_{\Omega^4_{1\oplus 35}(M)})}{\Ima(\ddphi)} = \cfrac{K}{\Ima(\ddphi)}.
\end{align}
\end{definition}

As a result of \Cref{thm:Spin7Hodge} we get the following
\begin{lemma}\label{lem:Spin7coh}
Given a compact Spin(7)-manifold $(M^8, \Phi)$, the linear map 
\begin{align*}
 H^{\Phi} \rightarrow H^4(M, \bR),\ \ \ [\gamma]_{\Phi} \mapsto [\gamma]   \end{align*}
 is injective and defines an isomorphism
 \begin{align*}
H^{\Phi}(M)\cong \cH^4_1(M) \oplus \cH^4_{35}(M).                    
 \end{align*} \qed\end{lemma}

\medskip

If we denote the \emph{refined Betti numbers} $b^k_l=\dim(H^k_l(M, \bR))$, then for a compact Spin(7)-manifold we see that 
\begin{align*}
\dim H^{\Phi}(M) = b^4_1+b^{4}_{35}.
\end{align*}
%If in addition, $M$ has full holonomy Spin(7) 
In fact, it follows from \cite[\textsection 10.6]{joycebook} that
\begin{align*}
\cH^4_1(M)=\bR\langle \Phi \rangle,\ \ \cH^4_{35}(M) = \cH^4_{-}(M),
\end{align*}
where $\cH^4_{-}(M)$ is the space of harmonic anti-self-dual $4$-forms on $M$. Thus, we have the formula
\begin{align*}
\dim(H^{\Phi}(M)) = 1+b^4_{-}(M)=1+ \cfrac{b^4(M)-\sigma(M)}{2},
%\text{for\ }M\ \text{with\ full\ holonomy},
\end{align*}
where $\sigma(M)$ is the signature of $M$.

\medskip

For example, if we take the flat torus $\mathbb{T}^8$ with the standard torsion-free Spin(7)-structure $\Phi_0$, then since all constant forms are harmonic, we get that
\begin{align*}
b^4(\mathbb{T}^8) = 70
\end{align*}
and 
\begin{align*}
\dim H^{\Phi_0}(\mathbb{T}^8) = \dim \cH^4_1(\mathbb{T}^8)+\dim \cH^4_{35}(\mathbb{T}^8) = 1+35=36.
\end{align*}
This simple example shows that the $\ddphi$-cohomology is a Spin(7)-refinement of the de Rham cohomology. In a similar manner, we can explcitly compute the dimension of the $\ddphi$-cohomology space if we have the information of the Betti numbers. We again consider an explicit example. 

\medskip

Consider the Spin(7)-manifold $(M^8, \Phi)$ with full holonomy Spin(7) from \cite[Table 14.1, p.380]{joycebook} whose Betti numbers are $(b^2, b^3, b^4)=(12, 16, 150)$. From \cite[Thm. 10.6.1]{joycebook} we have the formula
\begin{align*}
24\hat{A}(M) = -1+b^1-b^2+b^3+b^4_{+}-2b^4_{-},
\end{align*}
where $\hat{A}(M)$ is the $A$-hat genus of $M$. Since for a holonomy Spin(7) $M$, $\hat{A}(M)=1$, $b^1=0$ and $b^4_+ = b^4-b^4_{-}$, we get
\begin{align*}
b^4_{-}=\cfrac{b^4-b^2+b^3-25}{3},    
\end{align*}
and thus for the example considered here, we have
\begin{align*}
b^4_{-} = 43.
\end{align*}
We therefore get 
\begin{align*}
\dim H^{\Phi}(M) = 1+43=44.
\end{align*}
Similarly, we can explicitly compute the dimension of the $\ddphi$-cohomology provided we have information about the Betti numbers of the manifold. In fact, if the torsion-free Spin(7)-manifold is of the form $M=N\times S^1$, where $N$ is a compact torsion-free $\G2$-manifold, then we can relate the information about $H^{\Phi}(M)$ in terms of corresponding Bott--Chern-type cohomology space for $N$. To describe this, we have the following proposition.

\begin{proposition}
Let $(N^7,\g2)$ be a compact torsion-free $\G2$-manifold and let $\psi=*_\varphi\varphi$. Consider the $8$-manifold $M=N\times S^1$ with $\Phi=dt\wedge\g2+\psi$. If $f\in C^\infty(N\times S^1)$ and $t$ is the coordinate on $S^1$, we write
\begin{align*}
\nabla^M f=X+f_t\partial_t,\qquad X=\nabla^N(f(\,\cdot\,,t)).
\end{align*}
Then
$d^\Phi f= f_t\varphi-dt\wedge(X\lrcorner\varphi)+X\lrcorner\psi$ and
\begin{align*}
\ddphi f= d_Nf_t\wedge\varphi+d_N(X\lrcorner\psi)+ dt\wedge \left(f_{tt}\varphi
+d_N(X\lrcorner\varphi) +\partial_t(X\lrcorner\psi)\right).
\end{align*}
\end{proposition} 

\begin{proof}
Using $\Phi=dt\wedge\varphi+\psi$ and $\nabla^Mf=X+f_t\partial_t$, we get
\begin{align*}
(\nabla^Mf)\lrcorner\Phi &= X\lrcorner(dt\wedge\varphi) +X\lrcorner\psi +f_t\partial_t\lrcorner(dt\wedge\varphi)\\
&= -dt\wedge(X\lrcorner\varphi) +X\lrcorner\psi +f_t\varphi.
\end{align*}
Hence, using the fact that $d_M=d_N+dt\wedge\partial_t$, and the torsion-freeness of $(N, \varphi)$, $d_N\varphi=d_N\psi=0$, gives the stated expression. 
\end{proof}

Thus, we see that if $f$ is independent of $t$, then
\begin{align*}
\ddphi f = dt\wedge d_N\bigl((\nabla^Nf)\lrcorner\varphi\bigr)
+ d_N\bigl((\nabla^Nf)\lrcorner\psi\bigr).
\end{align*}
In other words, if we write $d^\varphi f=(\nabla^Nf)\lrcorner\varphi$ and $d^\psi f=(\nabla^Nf)\lrcorner\psi$, then
\begin{align*}
dd^\Phi f=dt\wedge dd^\varphi f+dd^\psi f.
\end{align*}
This formula gives a direct relation between the Spin(7) $\ddphi$-operator and the associated Harvey–Lawson operators $dd^\varphi, dd^\psi$ on the underlying $\G2$-manifold and we get the following corollary.

\begin{corollary}
There is a natural isomorphism
\begin{align*}
H^\Phi(N\times S^1) \cong \mathbb R[\Phi]\oplus H^3(N,\mathbb R)
\end{align*}
given on harmonic representatives by
\begin{align*}
(a,[\beta]) \longmapsto a[\Phi]+[dt\wedge\beta-*_7\beta].
\end{align*}
In particular,
\begin{align*}
\dim H^\Phi(N\times S^1)=1+b_3(N).
\end{align*}
\qed
\end{corollary}

{\textbf{Local $\ddphi$-lemma is false.}} It is a natural question that whether a \emph{local} $\ddphi$-lemma is true for compact Spin(7)-manifolds. A local $dd^\Phi$-lemma does not hold, in general, on compact torsion-free Spin(7)-manifolds. More precisely, suppose that $(M^8,\Phi)$ is compact and non-flat. We claim that there is a point $p\in M$ such that $\Phi$ is not $dd^\Phi$-exact on any neighbourhood of $p$. Indeed, since $\Phi$ is torsion-free, we know that
\begin{align*}
dd^\Phi f=(\operatorname{Hess}^g f)\diamond\Phi.
\end{align*}
Since the map
\begin{align*}
S^2 \longrightarrow\Omega^4_1\oplus\Omega^4_{35},
	\qquad A\longmapsto A\diamond\Phi,
\end{align*}
is injective, and $g\diamond\Phi=4\Phi$, we get that
\begin{align*}
dd^\Phi f=\Phi \quad\Longleftrightarrow\quad
	\operatorname{Hess}^g f=\frac14 g.
\end{align*}
Suppose that such a potential exists locally near every point. On each
such neighbourhood set $X=4\nabla f$. Then
\begin{align*}
\nabla X={\rm{Id}},\qquad \mathcal L_Xg=2g.
\end{align*}
If $\Theta_t$ is the local flow of $X$, then $\Theta_t^*g=e^{2t}g$. Consequently, for $u=|{\rm{Rm}}(g)|_g^2$,
\begin{align*}
\Theta_t^*u=e^{-4t}u, \qquad\text{and hence}\qquad X(u)=-4u.
\end{align*}
Choose a point $p$ where $u$ attains its maximum, which exists as $M$ is compact. Since $du_p=0$, the last identity gives $u(p)=0$. It follows that $u\equiv0$, contradicting the non-flatness of $g$.
	
Thus $\Phi$ is not locally $dd^\Phi$-exact near some point. On a sufficiently small contractible neighbourhood it is nevertheless
$d$-exact, by the Poincaré lemma, and therefore gives a counterexample
to a local $dd^\Phi$-lemma.

\subsection{Cayley-positive cones and Spin(7)-potentials}\label{subsec:cayleypositive}

Using our $\ddphi$-lemma, we provide natural analogues of the notions of K\"ahler cones and K\"ahler potentials in Spin(7)-geometry. We first start with some corollaries of the $\ddphi$-lemma.

\begin{corollary}
Let $(M^8, \Phi)$ be a compact Spin(7)-manifold and let $\gamma\in \Omega^4_{1\oplus 35}(M)$ be closed. Then the equation
\begin{align}
\ddphi f=\gamma
\end{align}
admits a smooth solution if and only if the de Rham class $[\gamma]\in H^4(M, \bR)$ vanishes. Thus, the $\ddphi$-cohomology is the obstruction for solving the Harvey--Lawson type Hessian equation.
\end{corollary}

\begin{proof}
Suppose $[\gamma]=0$ in $H^4(M, \bR)$. The hypothesis $d\gamma=0$ implies that $\gamma$ is exact hence from \Cref{lem:Spin7delbarlemma} we see that $\gamma\in \Ima(\ddphi)$ and hence there exists a smooth function $f$ on $M$ such that $\gamma=\ddphi f$. Conversely, if $\gamma=\ddphi f$ for some $f\in C^{\infty}(M)$ then $\gamma$ is exact and it being closed as well implies that $[\gamma]=0\in H^4(M, \bR)$.
\end{proof}

Since on a compact Spin(7)-manifold we have $dd^\Phi f =d\bigl((\nabla f)\lrcorner\Phi\bigr) \stackrel{(\ref{eq:imddphi1})}{=} \cL_{\del f} \Phi$, we see that $dd^\Phi$-exact variations are precisely infinitesimal variations of $\Phi$ generated by gradient vector fields. For a compact torsion-free Spin(7)-manifold, the infinitesimal moduli space modulo diffeomorphisms isotopic to the identity is (cf. \cite[Thm. 10.7.1]{joycebook}
\begin{align*}
T_{[\Phi]}\mathcal M_{\operatorname{Spin}(7)} \cong
\mathcal H^4_1\oplus\mathcal H^4_7\oplus\mathcal H^4_{35}.
\end{align*}
and since $H^\Phi(M) \cong \mathcal H^4_1\oplus\mathcal H^4_{35}$, there is a natural exact sequence
\begin{align*}
0\longrightarrow H^\Phi(M) \longrightarrow T_{[\Phi]}\mathcal M_{\operatorname{Spin}(7)}\longrightarrow \mathcal H^4_7(M)\longrightarrow 0.
\end{align*}
Thus $H^\Phi(M)$ captures all infinitesimal torsion-free Spin(7)-deformations except the $\Omega^4_7$-directions. Since for compact full holonomy Spin(7)-manifolds we have
$\cH^4_7(M)=\{0\}$, we get the following corollary.

\begin{corollary}
Let $(M^8, \Phi)$ be a compact Spin(7)-manifold with full holonomy Spin(7). Then
\begin{align*}
H^{\Phi}(M) \cong T_{[\Phi]}\cM_{\text{Spin(7)}},
\end{align*}
where the right-hand-side is the the tangent space at $[\Phi]$ of the moduli space of torsion-free Spin(7)-structures modulo diffeomorphisms isotopic to the identity. 
\end{corollary}

We now see an application of $\ddphi$-cohomology to calibrated geometry. Recall that a $4$-dimensional submanifold $C^4\subset M^8$ is called \textbf{Cayley} is it is calibrated by the $4$-form $\Phi$. Thus,
\begin{align*}
\Phi\mid_C = \vol_C.
\end{align*}

\begin{corollary}
Suppose $(M^8, \Phi)$ is a compact Spin(7)-manifold and $C^4\subset M^8$ is a compact, oriented Cayley submanifold without boundary in $M$. The map
\begin{align*}
P_C: H^{\Phi}(M) \rightarrow \bR,\ \ P_C([\gamma]) = \int_C \gamma,
\end{align*}
is a well-defined linear functional on $H^{\Phi}(M)$. Thus, the Cayley period map $P_C$ factors through the $\ddphi$-cohomology space $H^{\Phi}(M)$. 
\end{corollary}
\begin{proof}
If $[\gamma]=[\gamma']\in H^{\Phi}(M)$ then there exists $f\in C^{\infty}(M)$ such that $\gamma'=\gamma + \ddphi f$. Integrating over closed $C$ and using Stokes's theorem implies $\int_C \gamma'=\int_C\gamma$ and hence the map $P_C$ is well-defined.
\end{proof}
In particular, if $\int_C \gamma \neq 0$ for some Cayley cycle $C$, then $\gamma$ cannot be globally $\ddphi$-exact.

\medskip

Given that the $\ddphi$-cohomology is an analogue of the Bott--Chern cohomology in K\"ahler geometry, we expect it to define \emph{Spin(7)-cones} in the same way as K\"ahler cones are defined. Up to the knowledge of the authors, such notions have not been considered before for Spin(7)-geometry. We provide this below using our $\ddphi$-cohomology. Some of these are implicit in the seminal work of Harvey--Lawson \cite{HL-intropotential}. We first make the following

\begin{definition}\label{def:cayleypositive}
Let $(M^8, \Phi)$ be a Spin(7)-manifold. 
%and let $P^4\subset M^8$ be a Cayley submanifold. 
A $4$-form $\gamma\in \Omega^4(M)$ is called \textbf{Cayley-positive} if for every $x\in M$ and every oriented Cayley $4$-plane $P\subset T_xM$, we have
\begin{align}\label{eq:cayleypositiveeqn}
\gamma_x(P)\geq 0.    
\end{align}
It is called \textbf{strictly Cayley-positive} if $\gamma_x(P)>0$.
\end{definition}

The following local result was proved in \cite[Corr. 2.5]{HL-intropotential}. If $(M^8, \Phi)$ is a Spin(7)-manifold then for any $f\in C^{\infty}(M)$ and Cayley submanifold $C^4\subset M^8$ we have
\begin{align}\label{eq:hlresult}
(\ddphi f)\mid_C = -\Delta_C(f\mid_C)\vol_C.
\end{align}
Thus, if we define \textbf{$\Phi$-plurisubharmonic functions} ($\Phi$-psh) to be those $f\in C^{\infty}(M)$ such that the $4$-form $\ddphi f$ is Cayley-positive in the sense of \Cref{def:cayleypositive}, then as showed in \cite{HL-intropotential}, a $\Phi$-psh function on a compact Spin(7)-manifold must be constant. This follows from \cref{eq:hlresult} and the maximum principle. The next result uses the $\ddphi$-cohomology and can be considered as an analogue of the fact that there are no non-zero exact, positive $(1,1)$-forms on a compact K\"ahler manifold.

\begin{lemma}\label{lem:caypos}
Let $(M^8, \Phi)$ be a compact Spin(7)-manifold. Let $\gamma\in \Omega^4_{1\oplus 35}(M)$ be a closed and Cayley-positive $4$-form. If $\gamma$ is de Rham exact then $\gamma=0$.
\end{lemma}

\begin{proof}
Since $\gamma$ is closed and exact, the Spin(7) $\ddphi$-lemma \Cref{lem:Spin7delbarlemma} shows that $\gamma=\ddphi f$ for some $f\in C^{\infty}(M)$. Since $\gamma$ is Cayley-positive, we get that $f$ is $\Phi$-psh and hence $f$ must be constant. This proves that $\gamma=\ddphi f=0$.     
\end{proof}

We now define the analogues of K\"ahler cone for Spin(7)-manifolds\footnote{T. Pacini has informed the authors that analogous notions for $\G2$-manifolds have been studied by him and A. Raffero.}.

\begin{definition}\label{def:spin7cone}
Given a Spin(7)-manifold $(M^8, \Phi)$, define the \emph{Cayley-positive cone} $\cC^{\Phi}(M) \subset H^{\Phi}(M)$ by
\begin{align}
\cC^{\Phi}(M) = \left\{ [\gamma]_{\Phi}\in H^{\Phi}(M)\ :\ \gamma\ \text{has\ a\ closed\ Cayley-positive\ representative}\right\}.
\end{align}
Similarly, the \textbf{strict Cayley-positive cone} $\cK^{\Phi}(M)$ is defined as
\begin{align}
\cK^{\Phi}(M) = \left\{ [\gamma]_{\Phi}\in H^{\Phi}(M)\ :\ \gamma\ \text{has\ a\ closed\ strictly\ Cayley-positive\ representative}\right\}.
\end{align}
\end{definition}
In particular, the space $\cK^{\Phi}(M)$ should be seen as the analogue of the K\"ahler cone. Note that since both $\cC^{\Phi}(M)$ and $\cK^{\Phi}(M)$ are defined using the $\ddphi$-cohomological space $H^{\Phi}(M)$, it's understood that $4$-forms in either of these cones are in $\Omega^4_1(M)\oplus \Omega^4_{35}(M)$. We prove some basic properties of the spaces $\cC^{\Phi}(M)$ and $\cK^{\Phi}(M)$ in the following proposition.

\begin{proposition}\label{prop:conesprop}
Let $(M^8,\Phi)$ be a compact Spin(7)-manifold. Then
\begin{enumerate}[(1)]
    \item \label{part1cones} The space $\cC^{\Phi}(M)$ is a pointed convex cone.
    \item \label{part2cones} The space $\cK^{\Phi}(M)$ is an open convex cone. It is open in the finite-dimensional topology induced by the isomorphism $H^{\Phi}(M)\cong \cH^4_1(M)\oplus \cH^4_{35}(M)$ in \Cref{lem:Spin7coh}.
    \item \label{part3cones} The Spin(7)-structure $\Phi$ satisfies $[\Phi]_{\Phi}\in \cK^{\Phi}(M)$.
\end{enumerate}
\end{proposition}

\begin{proof}
The space $\cC^{\Phi}(M)$ is claerly a cone. It is convex because if $s, t\geq 0$ and $\gamma_1, \gamma_2\in \cC^{\Phi}(M)$ then by \cref{eq:cayleypositiveeqn} $s\gamma_1+t\gamma_2$ is Cayley-positive. To prove that it is also pointed, we will need to show that $\cC^{\Phi}(M)\cap -\cC^{\Phi}(M)=\{0\}$. Let $[\gamma]_{\Phi}\in \cC^{\Phi}(M)\cap -\cC^{\Phi}(M)$. By \Cref{def:spin7cone} this means that there are closed Cayley-positive representative $\gamma_+$ and $\gamma_-$ such that $[\gamma_+]_{\Phi}=-[\gamma_-]_{\Phi}$ and since hence $[\gamma_++\gamma_-]_{\Phi}=0$. By the $\ddphi$-cohomology in \Cref{lem:Spin7coh} we get that
\begin{align*}
\gamma_++\gamma_-=\ddphi f,\ f\in C^{\infty}(M).
\end{align*}
We see from \Cref{lem:caypos} that $\gamma_++\gamma_-=0$ and hence $\gamma_-=-\gamma_+$. Since both $\gamma_+$  and $\gamma_-$ are Cayley-positive, this implies that for all $x\in M$ and all Cayley plane $P\subset T_xM$ we have $(\gamma_+)_x(P)\geq 0$ and $(-\gamma_+)_x(P)\geq 0$, thus forcing $(\gamma_+)_x(P)=0$. Thus, the $4-$form $\gamma_+\in \Omega^4_{1\oplus 35}(M)$ vanishes on all Cayley planes at all points and hence from the standard linear algebra of Spin(7)-structures (see \Cref{prop:cayleyplanevanish}) we see that the form $\gamma_+=0$ on $M$. Similarly, $\gamma_-=0$ on $M$. Thus, $\cC^{\Phi}(M)\cap -\cC^{\Phi}(M)=\{0\}$ and $\cC^{\Phi}(M)$ is a pointed convex cone which proves \Cref{part1cones}.

\medskip

\Cref{part2cones} follows from the fact that strict positivity is an open condition and since $M$ is compact, a strictly Cayley-positive representative remains strictly Cayley-positive under small perturbations. Indeed, if $[\gamma]_\Phi\in \cK^\Phi(M)$ we denote a strictly Cayley-positive representative by the same symbol $\gamma$. Since the Cayley Grassmannian bundle over $M$ is compact, there exists $\varepsilon>0$ such that $\gamma(P)\geq\varepsilon$ for every Cayley $4$-plane $P\subset T_xM$ for all $x\in M$. Under the isomorphism $H^\Phi(M)\cong \cH^4_1\oplus\cH^4_{35}$ in \Cref{lem:Spin7coh}, sufficiently small classes have harmonic representatives $h$ with $\|h\|_{C^0}<\frac{\varepsilon}{2}$. Then $\gamma+h$ is still closed, lies in $\Omega^4_1\oplus\Omega^4_{35}$ and strictly Cayley-positive. Hence $[\gamma]_\Phi+[h]_\Phi\in \cK^\Phi(M)$. Therefore $\cK^\Phi(M)$ is open.

\medskip

Finally, since $\Phi$ calibrates Cayley submanifolds, hence for any Cayley plane $P$ we have $\Phi(P)=1$ which implies that $\Phi$ is strictly Cayley-positive. Since $\Phi$ is torsion-free hence $d\Phi=0$ and trivially $\Phi\in \Omega^4_1(M)$ and hence $[\Phi]_{\Phi}\in \cK^{\Phi}(M)$, which proves \Cref{part3cones}.
\end{proof}

%%Linear-algebra fact proof - I am omitting it since it seems standard and maybe in the notes of Salamon--Walpuski and so might be too elementary for the paper. Let's write it here in case someone wants to see the details. 

We now give a proof of the fact that if $\gamma\in \Omega^4_{1\oplus 35}$ vanishes on every Cayley-plane then $\gamma=0$. This was used in the proof of \Cref{prop:conesprop}~\Cref{part1cones}. Although the result is standard and the proof is linear algebraic, we include it here for completeness sake. 

\begin{proposition}\label{prop:cayleyplanevanish}
Let $(V^8,\Phi_0)$ be the standard Spin(7)-vector space isomorphic to $\bR^8$. If
\begin{align*}
 \gamma\in \Lambda^4_1V^*\oplus \Lambda^4_{35}V^*   
\end{align*}
vanishes on every Cayley $4$-plane, then $\gamma=0$.    
\end{proposition}
\begin{proof} 
Let $\operatorname{Cay}(V)$ denote the set of oriented Cayley $4$-planes in $V$. Consider the linear map 
\begin{align*}
L:\Lambda^4_1V^*\oplus \Lambda^4_{35}V^*
\longrightarrow C^\infty(\operatorname{Cay}(V)),\ \ L(\gamma).(P)=\gamma(P)     
\end{align*} 
Since the group Spin(7) preserves the set of Cayley planes, the kernel of $L$ is Spin(7)-invariant. From the decomposition of forms,  $\Lambda^4_1$ and $\Lambda^4_{35}$ are irreducible Spin(7)-representations. Hence any Spin(7)-invariant subspace of $\Lambda^4_1\oplus \Lambda^4_{35}$ is a direct sum of some of these irreducible summands. Now $\Lambda^4_1$ is not contained in $\ker L$, since $\Phi_0(P)=1$ for every Cayley plane $P$ by definition. Also $\Lambda^4_{35}$ is not contained in $\ker E$. Indeed, take a standard oriented orthonormal basis $e^1,\ldots,e^8$ such that $P_0=\operatorname{span}\{e_1,e_2,e_3,e_4\}$ is a Cayley plane. Then the $4$-form $\gamma_0=e^{1234}-e^{5678}$ is anti-self-dual, hence lies in $\Lambda^4_-=\Lambda^4_{35}$ by \cref{eq:sdasdforms}. But $\gamma_0(P_0)=1.$ Therefore, $\Lambda^4_{35}$ is not contained in $\ker L$. Thus $\ker L$ contains neither irreducible summand $\Lambda^4_1$ nor $\Lambda^4_{35}$. Hence $\ker L=0$ and consequently $\gamma=0$. 
\end{proof}

A classical result in K\"ahler geometry states that K\"ahler classes are paired to positive complex subvarieties. Complex subvarieties are calibrated objects in K\"ahler geometry and hence we get the following result whose proof is obvious from the definitions and properties of (strict) Cayley-positive forms.

\begin{proposition}
Let $(M^8, \Phi)$ be a compact Spin(7)-manifold and $C^4\subset M^8$ be a compact Cayley submanifold. Then
\begin{align*}
[\gamma]_{\Phi}\in \cC^{\Phi}(M) \implies \int_C \gamma \geq 0.
\end{align*}
If in addition, $C\neq \emptyset$ then
\begin{align*}
 [\gamma]_{\Phi}\in \cK^{\Phi}(M) \implies \int_C \gamma > 0.   
\end{align*}
\qed
\end{proposition}

The next result describes \textbf{Spin(7)-potentials} which can be seen as an analogue of K\"ahler potentials.

\begin{lemma}\label{lem:spin7potential}
Let $(M^8, \Phi)$ be a compact Spin(7)--manifold. Let $\gamma_0\in \Omega^4_1(M)\oplus \Omega^4_{35}(M)$ be a fixed closed form. Let $\gamma \in \in \Omega^4_1(M)\oplus \Omega^4_{35}(M)$ be another closed form in the same de Rham class as $\gamma_0$. that is, $[\gamma]=[\gamma_0]\in H^4(M, \bR)$. Then there exists $f\in C^{\infty}(M)$ which is unique upto additive constants such that
\begin{align*}
\gamma=\gamma_0+\ddphi f.
\end{align*}
Thus, the Cayley-positive representatives of $[\gamma_0]_{\Phi}$ are parametrized by the space
\begin{align*}
\left\{f\in C^{\infty}(M)\ :\ \gamma_0+\ddphi f\ \text{is\ Cayley-positive} \right\}/\bR.
\end{align*}
\end{lemma}

\begin{proof}
Since $\gamma$ and $\gamma_0$ are closed and $[\gamma]=[\gamma_0]$, the form $\gamma-\gamma_0$ is exact. The $\ddphi$-lemma \Cref{lem:Spin7delbarlemma} implies that there exists $f\in C^{\infty}(M)$ such that
\begin{align*}
\gamma=\gamma_0+\ddphi f,
\end{align*}
which proves the existence part. If there are two functions $f$ and $h$ which satisfy the previous equation then
\begin{align*}
\ddphi f=\ddphi h \implies \ddphi(f-h)=0 \implies \pi^4_{1}(\ddphi(f-h))=0\implies \Delta (f-h)=0,
\end{align*}
and hence by compactness of $M$ we get that $f-h=\text{constant}$. Thus, the potential $f$ is unique up to constants and this also proves the parametrization of the Cayley-positive representatives of $[\gamma_0]_{\Phi}$.
\end{proof}

\subsection{Future questions}\label{subsec:futureques}
We end the paper with some interesting questions. Probably the most natural and important question is to find the analogue of the $\ddphi$-lemma for an $8$-manifold $M^8$ with arbitrary Spin(7)-structure. Note that \cref{eq:hlhessiangeneral} gives an expression of the Harvey--Lawson Hessian operator for an arbitrary Spin(7)-structure. 

\begin{question}
Find an operator $D^{\Phi}_T: C^{\infty}(M)\rightarrow \Omega^4_1(M)\oplus \Omega^4_{35}(M)$ such that
\begin{align*}
(\Omega^4_1(M)\oplus \Omega^4_{35}(M)) \cap \ker d_T = \Ima(D^{\Phi}_T)\oplus \cH^{\Phi}_T,
\end{align*}
for some suitable first-order differential operator $d_T$ and "harmonic sections" $\cH^{\Phi}_T$.
\end{question}
Our \Cref{table:dwithtorsion}, which is true for any Spin(7)-structure, should be the starting point for such considerations.

\begin{question}
Can we say more about the properties of the Cayley-positive cones $\cC^{\Phi}(M)$ and $\cK^{\Phi}(M)$ for manifolds with either torsion-free or arbitrary Spin(7)-structrues? In particular, an analogue of \Cref{lem:spin7potential} can provide ways to formulate the analogue of Calabi-type conjecture in Spin(7)-geometry.
\end{question}

\begin{question}
\emph{$\ddphi$-cohomology and geometric flows of Spin(7)-structures.}
\end{question}
One of the applications of Bott-Chern cohomology in K\"ahler geometry is to define natural locus of K\"ahler forms in which to study various geometric flows of K\"ahler metrics. There have been proposals for deforming an arbitrary Spin(7)-structure to a torsion-free one by means of parabolic partial differential equations. The geometric flows for which we have a well-posed theory on compact manifolds are the isometric/harmonic flow of Spin(7)-structures \cite{dle-isometric}, the negative gradient flow of the $L^2$-norm of the torsion functional \cite{Spin7-flow_Dwivedi} and the Ricci-harmonic flow of Spin(7)-structures \cite[\textsection 7]{dwivedi-rhf}. Thus, it would be interesting to find natural class of admissible $4$-forms using some form of \Cref{lem:Spin7delbarlemma} to study the long-time existence and convergence of the flows within that class. Since there is no analogue of Bryant's Laplacian flow of closed $\G2$-structures in the Spin(7)-case since $d\Phi=0$ implies $\Phi$ is torsion-free, it would be feasible to restrict the flows in \cite{Spin7-flow_Dwivedi} or \cite{dwivedi-rhf} to some sub-class of Spin(7)-structures rather than their current status of being studied on all Spin(7)-structures. This again would require the understanding of \Cref{lem:Spin7delbarlemma}  for arbitrary Spin(7)-structures. 

\printbibliography 

\noindent
	(SD): Fachbereich Mathematik, Universität Hamburg, Bundesstraße 55, 20146 Hamburg, Germany.\\
	\href{mailto:shubham.dwivedi@uni-hamburg.de}{shubham.dwivedi@uni-hamburg.de}\\
	
	\noindent
	(RS): University of Münster, Einsteinstrasse 62, 48149 Münster, Germany. \\
	\href{ rsinghal@uni-muenster.de}{ rsinghal@uni-muenster.de}.

\end{document}